\documentclass{amsart}
\pdfoutput=1
\usepackage[encapsulated]{CJK}
\usepackage{ucs}
\usepackage[utf8x]{inputenc}
\usepackage{amssymb,latexsym,graphicx}
\usepackage[section]{placeins}
\usepackage{enumerate}
\usepackage[T1]{fontenc}
\usepackage{mathrsfs}
\usepackage{algorithmic}
\usepackage[section]{algorithm}
\usepackage{pgf,tikz}
\usetikzlibrary{arrows,shapes,positioning,calc,matrix,intersections}
\usepackage{cite}
\usepackage{float}
\usepackage{placeins}
\usepackage{longtable}

\theoremstyle{plain}
\newtheorem{theorem}{Theorem}
\numberwithin{figure}{section}
\newtheorem{lemma}{Lemma}
\numberwithin{lemma}{section}
\newtheorem{proposition}[lemma]{Proposition}

\theoremstyle{definition}
\newtheorem{definition}[lemma]{Definition}

\theoremstyle{remark}

\newtheorem*{remark}{Remark}

\newtheorem*{example*}{Example}

\begin{document}\begin{CJK*}{UTF8}{min}
\title[Link diagram encoding]{A method of encoding generalized link diagrams}
\author{Chad Musick}

\begin{abstract}We describe a method of encoding various types of link diagrams, including those with classical, flat, rigid, welded, and virtual crossings. We show that this method may be used to encode link diagrams, up to equivalence, in a notation whose length is a cubic function of the number of 'riser marks'. For classical knots, the minimal number of such marks is twice the bridge index, and a classical knot diagram in minimal bridge form with bridge index $b$ may be encoded in space $\mathcal{O}(b^2)$. A set of moves on the notation is defined. As a demonstration of the utility of the notation we give another proof that the Kishino virtual knot is non-classical.
\end{abstract}

\maketitle

\section{Introduction}
A link $L$ may be viewed as an immersion of circles into some space $M$, the exact $M$ depending on the type of the link. Typically, we are interested in equivalence classes of links given by ambient isotopies in $M$; however, we often categorize these equivalence classes by looking at planar projections and defining equivalence by referring to a set of diagrams reachable from any good projection by a set of local changes.

In this article, we strive to achieve a balance between the global nature of ambient isotopies in space and the localized nature of diagrammatic moves. In order to do this for a specific link, we adopt a method that relies on a combinatorial sentence and looks at moves on partial diagrams arising from this sentence. A method is given for forming a combinatorial sentence from a diagram, and a corresponding method is given for drawing a diagram from the combinatorial sentence. This gives both a method of encoding links and a method of making moves on links. We also show how the new notation may be given in a fixed-length form. This is particularly helpful for those diagrams that have a large number of crossings or points around which the winding number of the curve is high.

Two applications of this notation and the associated moves are given. We provide an elementary proof that the Kishino virtual knot is non-classical, and we give a $3$-bridge presentation of the $11$-crossing knot $k11a1$. 

\section{Standard construction}

We wish to demonstrate a method of encoding link diagrams generally. Historically, link diagrams referred to a projection to $S^2$ or $\mathbb{R}^2$ of a union of disjoint circles embedded in $S^3$ or $\mathbb{R}^3$. However, our primary interest is in the diagrams as objects in their own right. An overview of the subject is given in Nelson \cite{nelson}. We therefore make the following definition.

\begin{definition}[link diagram]A \emph{link diagram} is a planar graph $G$ in $S^2$ that meets the following criteria:
\begin{enumerate}
\item each vertex is $4$-valent, and
\item in a neighborhood around each vertex, there is a pattern $($in, in, out, out$)$ to the orientation of edges, and
\item each of the vertices has been assigned one of the set of crossing types $\{$ classical +, classical -, flat, virtual, welded (up), welded (down), rigid. $\}$
\end{enumerate}
\end{definition}

\subsection{Crossing types}

We handle $6$ distinct crossing types, each of which has its own diagrammatic rules. These crossing types are classical, flat, rigid, welded (both up and down varieties), and virtual. Figure \ref{F:crossings} shows the graphical notation we use for each type. All strands are oriented, but orientations are only shown where they are relevant.

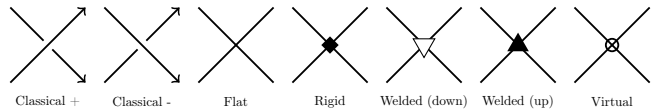
\begin{figure}[ht]
\begin{center}
\begin{tikzpicture}[every node/.style={scale=.5}, scale=.5]

\begin{scope}
	\draw[line width=.7pt,->] (-1, 1) -- (1, -1);
	\node[circle, fill=white,draw=none] at (0, 0) {};
	\draw[line width=.7pt,->] (-1, -1) -- (1, 1);
	\draw[line width=.7pt,->] (1.5, -1) -- (3.5, 1);
	\node[circle, fill=white,draw=none] at (2.5, 0) {};
	\draw[line width=.7pt,->] (1.5, 1) -- (3.5, -1);
	\node[rectangle, fill=none] at (0, -1.5) {Classical +};
	\node[rectangle, fill=none] at (2.5, -1.5) {Classical -};
\end{scope}

\begin{scope}[cm={1,0,0,1,(5,0)}]
	\draw[line width=.7pt] (-1, -1) -- (1, 1) (-1, 1) -- (1, -1);
	\node[rectangle, fill=none] at (0, -1.5) {Flat};
\end{scope}

\begin{scope}[cm={1,0,0,1,(7.5,0)}]
	\draw[line width=.7pt] (-1, -1) -- (1, 1) (-1, 1) -- (1, -1);
	\node[fill=black,diamond] at (0, 0) {};
	\node[rectangle, fill=none] at (0, -1.5) {Rigid};
\end{scope}

\begin{scope}[cm={1,0,0,1,(10,0)}]
	\draw[line width=.7pt] (-1, -1) -- (1, 1) (-1, 1) -- (1, -1);
	\node[fill=white, regular polygon, regular polygon sides=3, rotate=180, draw=black] at (0, 0) {};
	\node[rectangle, fill=none] at (0, -1.5) {Welded (down)};
\end{scope}

\begin{scope}[cm={1,0,0,1,(12.5,0)}]
	\draw[line width=.7pt] (-1, -1) -- (1, 1) (-1, 1) -- (1, -1);
	\node[fill=black,regular polygon, regular polygon sides=3] at (0, 0) {};
	\node[rectangle, fill=none] at (0, -1.5) {Welded (up)};
\end{scope}

\begin{scope}[cm={1,0,0,1,(15,0)}]
	\draw[line width=.7pt] (-1, -1) -- (1, 1) (-1, 1) -- (1, -1);
	\node[fill=none,line width=.7pt, draw=black,circle] at (0, 0) {};
	\node[rectangle, fill=none] at (0, -1.5) {Virtual};
\end{scope}

\end{tikzpicture}
\caption{The different types of crossings}\label{F:crossings}
\end{center}
\end{figure}

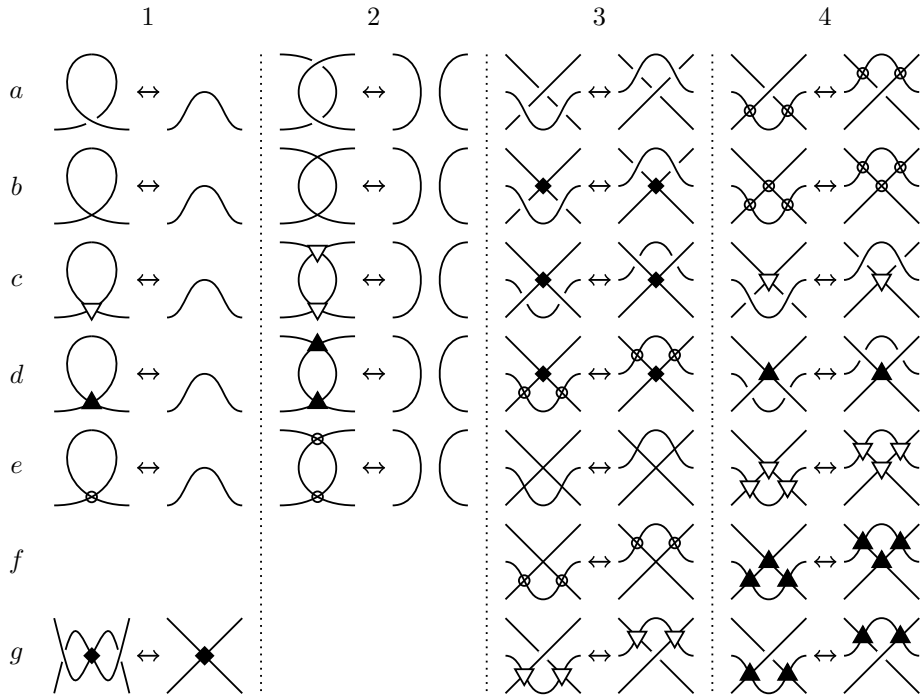
\begin{figure}[ht]
\begin{center}
\begin{tikzpicture}[every node/.style={scale=.5}, scale=.5, line width=.7pt]

	\begin{scope}[cm={1,0,0,1,(-1, 0)}]
	\node[style={scale=2}] at (0, 1) {$a$};
	\node[style={scale=2}] at (0, -1.5) {$b$};
	\node[style={scale=2}] at (0, -4) {$c$};
	\node[style={scale=2}] at (0, -6.5) {$d$};
	\node[style={scale=2}] at (0, -9) {$e$};
	\node[style={scale=2}] at (0, -11.5) {$f$};
	\node[style={scale=2}] at (0, -14) {$g$};
	\end{scope}

	\begin{scope}[cm={1,0,0,1,(-.5, 2)}]
	\node[style={scale=2}] at (3, 1) {$1$};
	\node[style={scale=2}] at (9, 1) {$2$};
	\node[style={scale=2}] at (15, 1) {$3$};
	\node[style={scale=2}] at (21, 1) {$4$};
	\end{scope}
	
	\begin{scope}[cm={1,0,0,1,(0,0)}]
	\begin{scope}[cm={1,0,0,1,(0,0)}]
	\draw[name path=a1, draw=none] (0, 0) .. controls (2, 0) and (2, 2) .. (1, 2);
	\draw[name path=a2, draw=none] (1, 2) .. controls (0, 2) and (0, 0) .. (2, 0);
	\draw (0, 0) .. controls (2, 0) and (2, 2) .. (1, 2);
	\draw[name intersections={of=a1 and a2}] (intersection-2) node[circle, draw=white,fill=white,inner sep=1mm] {};
	\draw (1, 2) .. controls (0, 2) and (0, 0) .. (2, 0);
	\draw (3, 0) .. controls (3.5, 0) and (3.5, 1) .. (4, 1) .. controls (4.5, 1) and (4.5, 0) .. (5, 0);
	\draw[<->] (2.2, 1) -- (2.8, 1);
	\end{scope}
	\begin{scope}[cm={1,0,0,1,(0,-2.5)}]
	\draw (1, 2) .. controls (0, 2) and (0, 0) .. (2, 0);
	\draw (0, 0) .. controls (2, 0) and (2, 2) .. (1, 2);
	\draw (3, 0) .. controls (3.5, 0) and (3.5, 1) .. (4, 1) .. controls (4.5, 1) and (4.5, 0) .. (5, 0);
	\draw[<->] (2.2, 1) -- (2.8, 1);
	\end{scope}
	\begin{scope}[cm={1,0,0,1,(0,-5)}]
	\draw[name path=a1, draw=none] (0, 0) .. controls (2, 0) and (2, 2) .. (1, 2);
	\draw[name path=a2, draw=none] (1, 2) .. controls (0, 2) and (0, 0) .. (2, 0);
	\draw (1, 2) .. controls (0, 2) and (0, 0) .. (2, 0);
	\draw (0, 0) .. controls (2, 0) and (2, 2) .. (1, 2);
	\draw (3, 0) .. controls (3.5, 0) and (3.5, 1) .. (4, 1) .. controls (4.5, 1) and (4.5, 0) .. (5, 0);
	\draw[name intersections={of=a1 and a2}] (intersection-2) node[regular polygon, regular polygon sides=3, rotate=180, draw=black,fill=white,inner sep=1mm] {};
	\draw[<->] (2.2, 1) -- (2.8, 1);
	\end{scope}
	\begin{scope}[cm={1,0,0,1,(0,-7.5)}]
	\draw[name path=a1, draw=none] (0, 0) .. controls (2, 0) and (2, 2) .. (1, 2);
	\draw[name path=a2, draw=none] (1, 2) .. controls (0, 2) and (0, 0) .. (2, 0);
	\draw (1, 2) .. controls (0, 2) and (0, 0) .. (2, 0);
	\draw[name intersections={of=a1 and a2}] (intersection-2) node[regular polygon, regular polygon sides=3, draw=black,fill=black,inner sep=1mm] {};
	\draw (0, 0) .. controls (2, 0) and (2, 2) .. (1, 2);
	\draw (3, 0) .. controls (3.5, 0) and (3.5, 1) .. (4, 1) .. controls (4.5, 1) and (4.5, 0) .. (5, 0);
	\draw[<->] (2.2, 1) -- (2.8, 1);
	\end{scope}
	\begin{scope}[cm={1,0,0,1,(0,-10)}]
	\draw[name path=a1, draw=none] (0, 0) .. controls (2, 0) and (2, 2) .. (1, 2);
	\draw[name path=a2, draw=none] (1, 2) .. controls (0, 2) and (0, 0) .. (2, 0);
	\draw (1, 2) .. controls (0, 2) and (0, 0) .. (2, 0);
	\draw[name intersections={of=a1 and a2}] (intersection-2) node[circle, draw=black,fill=none,inner sep=1mm] {};
	\draw (0, 0) .. controls (2, 0) and (2, 2) .. (1, 2);
	\draw (3, 0) .. controls (3.5, 0) and (3.5, 1) .. (4, 1) .. controls (4.5, 1) and (4.5, 0) .. (5, 0);
	\draw[<->] (2.2, 1) -- (2.8, 1);
	\end{scope}
	\begin{scope}[cm={1,0,0,1,(0,-15)}]
	\draw[name path=a1, draw=none] (0, 0) .. controls (.5, 2) .. (1, 1);
	\draw[name path=a2, draw=none] (0, 2) .. controls (.5, 0) .. (1, 1);
	\draw[name path=a3, draw=none] (1, 1) .. controls (1.5, 0) .. (2, 2);
	\draw[name path=a4, draw=none] (1, 1) .. controls (1.5, 2) .. (2, 0);
	\draw (0, 0) .. controls (.5, 2) .. (1, 1);
	\draw[name intersections={of=a1 and a2}] (intersection-1) node[circle, draw=white, fill=white, inner sep=1mm] {};
	\draw (0, 2) .. controls (.5, 0) .. (1, 1);
	\draw (1, 1) .. controls (1.5, 2) .. (2, 0);
	\draw[name intersections={of=a3 and a4}] (intersection-2) node[circle, draw=white, fill=white, inner sep=1mm] {}
		(intersection-1) node[diamond, draw=black, fill=black, inner sep=1mm] {};
	\draw (1, 1) .. controls (1.5, 0) .. (2, 2);
	\draw[<->] (2.2, 1) -- (2.8, 1);
	\draw[name path=a1] (3, 0) -- (5, 2);
	\draw[name path=a2] (3, 2) -- (5, 0);
	\draw[name intersections={of=a1 and a2}] (intersection-1) node[diamond, fill=black, draw=black] {};
	\end{scope}
	\end{scope}
	
	\begin{scope}[cm={1,0,0,1,(6,0)}]
	\begin{scope}[cm={1,0,0,1,(0,0)}]
	\draw[name path=a1, draw=none] (0, 0) .. controls (2, 0) and (2, 2) .. (0, 2);
	\draw[name path=a2, draw=none] (2, 0) .. controls (0, 0) and (0, 2) .. (2, 2);
	\draw (0, 0) .. controls (2, 0) and (2, 2) .. (0, 2);
	\draw[name intersections={of=a1 and a2}] (intersection-1) node[circle, draw=white, fill=white, inner sep=1mm] {}
		(intersection-2) node[circle, draw=white, fill=white, inner sep=1mm] {};
	\draw (2, 0) .. controls (0, 0) and (0, 2) .. (2, 2);
	\draw[<->] (2.2, 1) -- (2.8, 1);
	\draw (3, 0) .. controls (4, 0) and (4, 2) .. (3, 2);
	\draw (5, 0) .. controls (4, 0) and (4, 2) .. (5, 2);
	\end{scope}
	\begin{scope}[cm={1,0,0,1,(0,-2.5)}]
	\draw[name path=a1, draw=none] (0, 0) .. controls (2, 0) and (2, 2) .. (0, 2);
	\draw[name path=a2, draw=none] (2, 0) .. controls (0, 0) and (0, 2) .. (2, 2);
	\draw (2, 0) .. controls (0, 0) and (0, 2) .. (2, 2);
	\draw (0, 0) .. controls (2, 0) and (2, 2) .. (0, 2);
	\draw[<->] (2.2, 1) -- (2.8, 1);
	\draw (3, 0) .. controls (4, 0) and (4, 2) .. (3, 2);
	\draw (5, 0) .. controls (4, 0) and (4, 2) .. (5, 2);
	\end{scope}
	\begin{scope}[cm={1,0,0,1,(0,-5)}]
	\draw[name path=a1, draw=none] (0, 0) .. controls (2, 0) and (2, 2) .. (0, 2);
	\draw[name path=a2, draw=none] (2, 0) .. controls (0, 0) and (0, 2) .. (2, 2);
	\draw (2, 0) .. controls (0, 0) and (0, 2) .. (2, 2);
	\draw (0, 0) .. controls (2, 0) and (2, 2) .. (0, 2);
	\draw[name intersections={of=a1 and a2}] (intersection-1) node[regular polygon, regular polygon sides=3, rotate=180, draw=black, fill=white, inner sep=1mm] {}
		(intersection-2) node[regular polygon, regular polygon sides=3, rotate=180, draw=black, fill=white, inner sep=1mm] {};
	\draw[<->] (2.2, 1) -- (2.8, 1);
	\draw (3, 0) .. controls (4, 0) and (4, 2) .. (3, 2);
	\draw (5, 0) .. controls (4, 0) and (4, 2) .. (5, 2);
	\end{scope}
	\begin{scope}[cm={1,0,0,1,(0,-7.5)}]
	\draw[name path=a1, draw=none] (0, 0) .. controls (2, 0) and (2, 2) .. (0, 2);
	\draw[name path=a2, draw=none] (2, 0) .. controls (0, 0) and (0, 2) .. (2, 2);
	\draw (2, 0) .. controls (0, 0) and (0, 2) .. (2, 2);
	\draw (0, 0) .. controls (2, 0) and (2, 2) .. (0, 2);
	\draw[name intersections={of=a1 and a2}] (intersection-1) node[regular polygon, regular polygon sides=3, draw=black, fill=black, inner sep=1mm] {}
		(intersection-2) node[regular polygon, regular polygon sides=3, draw=black, fill=black, inner sep=1mm] {};
	\draw[<->] (2.2, 1) -- (2.8, 1);
	\draw (3, 0) .. controls (4, 0) and (4, 2) .. (3, 2);
	\draw (5, 0) .. controls (4, 0) and (4, 2) .. (5, 2);
	\end{scope}
	\begin{scope}[cm={1,0,0,1,(0,-10)}]
	\draw[name path=a1, draw=none] (0, 0) .. controls (2, 0) and (2, 2) .. (0, 2);
	\draw[name path=a2, draw=none] (2, 0) .. controls (0, 0) and (0, 2) .. (2, 2);
	\draw (2, 0) .. controls (0, 0) and (0, 2) .. (2, 2);
	\draw (0, 0) .. controls (2, 0) and (2, 2) .. (0, 2);
	\draw[name intersections={of=a1 and a2}] (intersection-1) node[circle, draw=black, fill=none, inner sep=1mm] {}
		(intersection-2) node[circle, draw=black, fill=none, inner sep=1mm] {};
	\draw[<->] (2.2, 1) -- (2.8, 1);
	\draw (3, 0) .. controls (4, 0) and (4, 2) .. (3, 2);
	\draw (5, 0) .. controls (4, 0) and (4, 2) .. (5, 2);
	\end{scope}
	\end{scope}
	
	\begin{scope}[cm={1,0,0,1,(12,0)}]
	\begin{scope}[cm={1,0,0,1,(0,0)}]
	\draw[name path=a1, draw=none] (0, 2) -- (2, 0);
	\draw[name path=a2, draw=none] (0, 0) -- (2, 2);
	\draw[name path=a3, draw=none] (0, 1) .. controls (.5, 1) and (.5, 0) .. (1, 0) .. controls (1.5, 0) and (1.5, 1) .. (2, 1);
	\draw (0, 2) -- (2, 0);
	\draw[name intersections={of=a1 and a2}] (intersection-1) node[circle, draw=white, fill=white, inner sep=1mm] {};
	\draw (0, 0) -- (2, 2);
	\draw[name intersections={of=a2 and a3}] (intersection-1) node[circle, draw=white, fill=white, inner sep=1mm] {};
	\draw[name intersections={of=a1 and a3}] (intersection-1) node[circle, draw=white, fill=white, inner sep=1mm] {};
	\draw (0, 1) .. controls (.5, 1) and (.5, 0) .. (1, 0) .. controls (1.5, 0) and (1.5, 1) .. (2, 1);
	\draw[name path=a1, draw=none] (3, 2) -- (5, 0);
	\draw[name path=a2, draw=none] (3, 0) -- (5, 2);
	\draw[name path=a3, draw=none] (3, 1) .. controls (3.5, 1) and (3.5, 2) .. (4, 2) .. controls (4.5, 2) and (4.5, 1) .. (5, 1);
	\draw (3, 2) -- (5, 0);
	\draw[name intersections={of=a1 and a2}] (intersection-1) node[circle, draw=white, fill=white, inner sep=1mm] {};
	\draw (3, 0) -- (5, 2);
	\draw[name intersections={of=a2 and a3}] (intersection-1) node[circle, draw=white, fill=white, inner sep=1mm] {};
	\draw[name intersections={of=a1 and a3}] (intersection-1) node[circle, draw=white, fill=white, inner sep=1mm] {};
	\draw (3, 1) .. controls (3.5, 1) and (3.5, 2) .. (4, 2) .. controls (4.5, 2) and (4.5, 1) .. (5, 1);
	\draw[<->] (2.2, 1) -- (2.8, 1);
	\end{scope}
	\begin{scope}[cm={1,0,0,1,(0,-2.5)}]
	\draw[name path=a1, draw=none] (0, 2) -- (2, 0);
	\draw[name path=a2, draw=none] (0, 0) -- (2, 2);
	\draw[name path=a3, draw=none] (0, 1) .. controls (.5, 1) and (.5, 0) .. (1, 0) .. controls (1.5, 0) and (1.5, 1) .. (2, 1);
	\draw (0, 2) -- (2, 0);
	\draw[name intersections={of=a1 and a2}] (intersection-1) node[diamond, draw=black, fill=black, inner sep=1mm] {};
	\draw (0, 0) -- (2, 2);
	\draw[name intersections={of=a2 and a3}] (intersection-1) node[circle, draw=white, fill=white, inner sep=1mm] {};
	\draw[name intersections={of=a1 and a3}] (intersection-1) node[circle, draw=white, fill=white, inner sep=1mm] {};
	\draw (0, 1) .. controls (.5, 1) and (.5, 0) .. (1, 0) .. controls (1.5, 0) and (1.5, 1) .. (2, 1);
	\draw[name path=a1, draw=none] (3, 2) -- (5, 0);
	\draw[name path=a2, draw=none] (3, 0) -- (5, 2);
	\draw[name path=a3, draw=none] (3, 1) .. controls (3.5, 1) and (3.5, 2) .. (4, 2) .. controls (4.5, 2) and (4.5, 1) .. (5, 1);
	\draw (3, 2) -- (5, 0);
	\draw[name intersections={of=a1 and a2}] (intersection-1) node[diamond, draw=black, fill=black, inner sep=1mm] {};
	\draw (3, 0) -- (5, 2);
	\draw[name intersections={of=a2 and a3}] (intersection-1) node[circle, draw=white, fill=white, inner sep=1mm] {};
	\draw[name intersections={of=a1 and a3}] (intersection-1) node[circle, draw=white, fill=white, inner sep=1mm] {};
	\draw (3, 1) .. controls (3.5, 1) and (3.5, 2) .. (4, 2) .. controls (4.5, 2) and (4.5, 1) .. (5, 1);
	\draw[<->] (2.2, 1) -- (2.8, 1);
	\end{scope}
	\begin{scope}[cm={1,0,0,1,(0,-5)}]
	\draw[name path=a1, draw=none] (0, 2) -- (2, 0);
	\draw[name path=a2, draw=none] (0, 0) -- (2, 2);
	\draw[name path=a3, draw=none] (0, 1) .. controls (.5, 1) and (.5, 0) .. (1, 0) .. controls (1.5, 0) and (1.5, 1) .. (2, 1);
	\draw (0, 1) .. controls (.5, 1) and (.5, 0) .. (1, 0) .. controls (1.5, 0) and (1.5, 1) .. (2, 1);
	\draw[name intersections={of=a2 and a3}] (intersection-1) node[circle, draw=white, fill=white, inner sep=1mm] {};
	\draw[name intersections={of=a1 and a3}] (intersection-1) node[circle, draw=white, fill=white, inner sep=1mm] {};
	\draw (0, 2) -- (2, 0);
	\draw[name intersections={of=a1 and a2}] (intersection-1) node[diamond, draw=black, fill=black, inner sep=1mm] {};
	\draw (0, 0) -- (2, 2);
	\draw[name path=a1, draw=none] (3, 2) -- (5, 0);
	\draw[name path=a2, draw=none] (3, 0) -- (5, 2);
	\draw[name path=a3, draw=none] (3, 1) .. controls (3.5, 1) and (3.5, 2) .. (4, 2) .. controls (4.5, 2) and (4.5, 1) .. (5, 1);
	\draw (3, 1) .. controls (3.5, 1) and (3.5, 2) .. (4, 2) .. controls (4.5, 2) and (4.5, 1) .. (5, 1);
	\draw[name intersections={of=a2 and a3}] (intersection-1) node[circle, draw=white, fill=white, inner sep=1mm] {};
	\draw[name intersections={of=a1 and a3}] (intersection-1) node[circle, draw=white, fill=white, inner sep=1mm] {};
	\draw (3, 2) -- (5, 0);
	\draw[name intersections={of=a1 and a2}] (intersection-1) node[diamond, draw=black, fill=black, inner sep=1mm] {};
	\draw (3, 0) -- (5, 2);
	\draw[<->] (2.2, 1) -- (2.8, 1);
	\end{scope}
	\begin{scope}[cm={1,0,0,1,(0,-7.5)}]
	\draw[name path=a1, draw=none] (0, 2) -- (2, 0);
	\draw[name path=a2, draw=none] (0, 0) -- (2, 2);
	\draw[name path=a3, draw=none] (0, 1) .. controls (.5, 1) and (.5, 0) .. (1, 0) .. controls (1.5, 0) and (1.5, 1) .. (2, 1);
	\draw (0, 1) .. controls (.5, 1) and (.5, 0) .. (1, 0) .. controls (1.5, 0) and (1.5, 1) .. (2, 1);
	\draw[name intersections={of=a2 and a3}] (intersection-1) node[circle, draw=black, fill=none, inner sep=1mm] {};
	\draw[name intersections={of=a1 and a3}] (intersection-1) node[circle, draw=black, fill=none, inner sep=1mm] {};
	\draw (0, 2) -- (2, 0);
	\draw[name intersections={of=a1 and a2}] (intersection-1) node[diamond, draw=black, fill=black, inner sep=1mm] {};
	\draw (0, 0) -- (2, 2);
	\draw[name path=a1, draw=none] (3, 2) -- (5, 0);
	\draw[name path=a2, draw=none] (3, 0) -- (5, 2);
	\draw[name path=a3, draw=none] (3, 1) .. controls (3.5, 1) and (3.5, 2) .. (4, 2) .. controls (4.5, 2) and (4.5, 1) .. (5, 1);
	\draw (3, 1) .. controls (3.5, 1) and (3.5, 2) .. (4, 2) .. controls (4.5, 2) and (4.5, 1) .. (5, 1);
	\draw[name intersections={of=a2 and a3}] (intersection-1) node[circle, draw=black, fill=none, inner sep=1mm] {};
	\draw[name intersections={of=a1 and a3}] (intersection-1) node[circle, draw=black, fill=none, inner sep=1mm] {};
	\draw (3, 2) -- (5, 0);
	\draw[name intersections={of=a1 and a2}] (intersection-1) node[diamond, draw=black, fill=black, inner sep=1mm] {};
	\draw (3, 0) -- (5, 2);
	\draw[<->] (2.2, 1) -- (2.8, 1);
	\end{scope}
	\begin{scope}[cm={1,0,0,1,(0,-10)}]
	\draw[name path=a1, draw=none] (0, 2) -- (2, 0);
	\draw[name path=a2, draw=none] (0, 0) -- (2, 2);
	\draw[name path=a3, draw=none] (0, 1) .. controls (.5, 1) and (.5, 0) .. (1, 0) .. controls (1.5, 0) and (1.5, 1) .. (2, 1);
	\draw (0, 1) .. controls (.5, 1) and (.5, 0) .. (1, 0) .. controls (1.5, 0) and (1.5, 1) .. (2, 1);
	\draw (0, 2) -- (2, 0);
	\draw (0, 0) -- (2, 2);
	\draw[name path=a1, draw=none] (3, 2) -- (5, 0);
	\draw[name path=a2, draw=none] (3, 0) -- (5, 2);
	\draw[name path=a3, draw=none] (3, 1) .. controls (3.5, 1) and (3.5, 2) .. (4, 2) .. controls (4.5, 2) and (4.5, 1) .. (5, 1);
	\draw (3, 1) .. controls (3.5, 1) and (3.5, 2) .. (4, 2) .. controls (4.5, 2) and (4.5, 1) .. (5, 1);
	\draw (3, 2) -- (5, 0);
	\draw (3, 0) -- (5, 2);
	\draw[<->] (2.2, 1) -- (2.8, 1);
	\end{scope}
	\begin{scope}[cm={1,0,0,1,(0,-12.5)}]
	\draw[name path=a1, draw=none] (0, 2) -- (2, 0);
	\draw[name path=a2, draw=none] (0, 0) -- (2, 2);
	\draw[name path=a3, draw=none] (0, 1) .. controls (.5, 1) and (.5, 0) .. (1, 0) .. controls (1.5, 0) and (1.5, 1) .. (2, 1);
	\draw (0, 1) .. controls (.5, 1) and (.5, 0) .. (1, 0) .. controls (1.5, 0) and (1.5, 1) .. (2, 1);
	\draw[name intersections={of=a2 and a3}] (intersection-1) node[circle, draw=black, fill=none, inner sep=1mm] {};
	\draw[name intersections={of=a1 and a3}] (intersection-1) node[circle, draw=black, fill=none, inner sep=1mm] {};
	\draw (0, 2) -- (2, 0);
	\draw (0, 0) -- (2, 2);
	\draw[name path=a1, draw=none] (3, 2) -- (5, 0);
	\draw[name path=a2, draw=none] (3, 0) -- (5, 2);
	\draw[name path=a3, draw=none] (3, 1) .. controls (3.5, 1) and (3.5, 2) .. (4, 2) .. controls (4.5, 2) and (4.5, 1) .. (5, 1);
	\draw (3, 1) .. controls (3.5, 1) and (3.5, 2) .. (4, 2) .. controls (4.5, 2) and (4.5, 1) .. (5, 1);
	\draw (3, 2) -- (5, 0);
	\draw (3, 0) -- (5, 2);
	\draw[name intersections={of=a2 and a3}] (intersection-1) node[circle, draw=black, fill=none, inner sep=1mm] {};
	\draw[name intersections={of=a1 and a3}] (intersection-1) node[circle, draw=black, fill=none, inner sep=1mm] {};
	\draw[<->] (2.2, 1) -- (2.8, 1);
	\end{scope}
	\begin{scope}[cm={1,0,0,1,(0,-15)}]
	\draw[name path=a1, draw=none] (0, 2) -- (2, 0);
	\draw[name path=a2, draw=none] (0, 0) -- (2, 2);
	\draw[name path=a3, draw=none] (0, 1) .. controls (.5, 1) and (.5, 0) .. (1, 0) .. controls (1.5, 0) and (1.5, 1) .. (2, 1);
	\draw (0, 1) .. controls (.5, 1) and (.5, 0) .. (1, 0) .. controls (1.5, 0) and (1.5, 1) .. (2, 1);
	\draw (0, 2) -- (2, 0);
	\node[draw=white,fill=white,inner sep=1mm] at (1, 1) {};
	\draw (0, 0) -- (2, 2);
	\draw[name intersections={of=a2 and a3}] (intersection-1) node[regular polygon, regular polygon sides=3, rotate=180, draw=black, fill=white, inner sep=1mm] {};
	\draw[name intersections={of=a1 and a3}] (intersection-1) node[regular polygon, regular polygon sides=3, rotate=180, draw=black, fill=white, inner sep=1mm] {};
	\draw[name path=a1, draw=none] (3, 2) -- (5, 0);
	\draw[name path=a2, draw=none] (3, 0) -- (5, 2);
	\draw[name path=a3, draw=none] (3, 1) .. controls (3.5, 1) and (3.5, 2) .. (4, 2) .. controls (4.5, 2) and (4.5, 1) .. (5, 1);
	\draw (3, 1) .. controls (3.5, 1) and (3.5, 2) .. (4, 2) .. controls (4.5, 2) and (4.5, 1) .. (5, 1);
	\draw (3, 2) -- (5, 0);
	\node[draw=white,fill=white,inner sep=1mm] at (4, 1) {};
	\draw (3, 0) -- (5, 2);
	\draw[name intersections={of=a2 and a3}] (intersection-1) node[regular polygon, regular polygon sides=3, rotate=180, draw=black, fill=white, inner sep=1mm] {};
	\draw[name intersections={of=a1 and a3}] (intersection-1) node[regular polygon, regular polygon sides=3, rotate=180, draw=black, fill=white, inner sep=1mm] {};
	\draw[<->] (2.2, 1) -- (2.8, 1);
	\end{scope}
	\end{scope}

	\begin{scope}[cm={1,0,0,1,(18,0)}]
	\begin{scope}[cm={1,0,0,1,(0,0)}]
	\draw[name path=a1, draw=none] (0, 2) -- (2, 0);
	\draw[name path=a2, draw=none] (0, 0) -- (2, 2);
	\draw[name path=a3, draw=none] (0, 1) .. controls (.5, 1) and (.5, 0) .. (1, 0) .. controls (1.5, 0) and (1.5, 1) .. (2, 1);
	\draw (0, 2) -- (2, 0);
	\draw[name intersections={of=a1 and a2}] (intersection-1) node[circle, draw=white, fill=white, inner sep=1mm] {};
	\draw (0, 0) -- (2, 2);
	\draw[name intersections={of=a2 and a3}] (intersection-1) node[circle, draw=black, fill=none, inner sep=1mm] {};
	\draw[name intersections={of=a1 and a3}] (intersection-1) node[circle, draw=black, fill=none, inner sep=1mm] {};
	\draw (0, 1) .. controls (.5, 1) and (.5, 0) .. (1, 0) .. controls (1.5, 0) and (1.5, 1) .. (2, 1);
	\draw[name path=a1, draw=none] (3, 2) -- (5, 0);
	\draw[name path=a2, draw=none] (3, 0) -- (5, 2);
	\draw[name path=a3, draw=none] (3, 1) .. controls (3.5, 1) and (3.5, 2) .. (4, 2) .. controls (4.5, 2) and (4.5, 1) .. (5, 1);
	\draw (3, 2) -- (5, 0);
	\draw[name intersections={of=a1 and a2}] (intersection-1) node[circle, draw=white, fill=white, inner sep=1mm] {};
	\draw (3, 0) -- (5, 2);
	\draw[name intersections={of=a2 and a3}] (intersection-1) node[circle, draw=black, fill=none, inner sep=1mm] {};
	\draw[name intersections={of=a1 and a3}] (intersection-1) node[circle, draw=black, fill=none, inner sep=1mm] {};
	\draw (3, 1) .. controls (3.5, 1) and (3.5, 2) .. (4, 2) .. controls (4.5, 2) and (4.5, 1) .. (5, 1);
	\draw[<->] (2.2, 1) -- (2.8, 1);
	\end{scope}
	\begin{scope}[cm={1,0,0,1,(0,-2.5)}]
	\draw[name path=a1, draw=none] (0, 2) -- (2, 0);
	\draw[name path=a2, draw=none] (0, 0) -- (2, 2);
	\draw[name path=a3, draw=none] (0, 1) .. controls (.5, 1) and (.5, 0) .. (1, 0) .. controls (1.5, 0) and (1.5, 1) .. (2, 1);
	\draw (0, 2) -- (2, 0);
	\draw[name intersections={of=a1 and a2}] (intersection-1) node[circle, draw=black, fill=none, inner sep=1mm] {};
	\draw (0, 0) -- (2, 2);
	\draw[name intersections={of=a2 and a3}] (intersection-1) node[circle, draw=black, fill=none, inner sep=1mm] {};
	\draw[name intersections={of=a1 and a3}] (intersection-1) node[circle, draw=black, fill=none, inner sep=1mm] {};
	\draw (0, 1) .. controls (.5, 1) and (.5, 0) .. (1, 0) .. controls (1.5, 0) and (1.5, 1) .. (2, 1);
	\draw[name path=a1, draw=none] (3, 2) -- (5, 0);
	\draw[name path=a2, draw=none] (3, 0) -- (5, 2);
	\draw[name path=a3, draw=none] (3, 1) .. controls (3.5, 1) and (3.5, 2) .. (4, 2) .. controls (4.5, 2) and (4.5, 1) .. (5, 1);
	\draw (3, 2) -- (5, 0);
	\draw (3, 0) -- (5, 2);
	\draw (3, 1) .. controls (3.5, 1) and (3.5, 2) .. (4, 2) .. controls (4.5, 2) and (4.5, 1) .. (5, 1);
	\draw[name intersections={of=a1 and a2}] (intersection-1) node[circle, draw=black, fill=none, inner sep=1mm] {};
	\draw[name intersections={of=a2 and a3}] (intersection-1) node[circle, draw=black, fill=none, inner sep=1mm] {};
	\draw[name intersections={of=a1 and a3}] (intersection-1) node[circle, draw=black, fill=none, inner sep=1mm] {};
	\draw[<->] (2.2, 1) -- (2.8, 1);
	\end{scope}
	\begin{scope}[cm={1,0,0,1,(0,-5)}]
	\draw[name path=a1, draw=none] (0, 2) -- (2, 0);
	\draw[name path=a2, draw=none] (0, 0) -- (2, 2);
	\draw[name path=a3, draw=none] (0, 1) .. controls (.5, 1) and (.5, 0) .. (1, 0) .. controls (1.5, 0) and (1.5, 1) .. (2, 1);
	\draw (0, 2) -- (2, 0);
	\draw (0, 0) -- (2, 2);
	\draw[name intersections={of=a2 and a3}] (intersection-1) node[circle, draw=white, fill=white, inner sep=1mm] {};
	\draw[name intersections={of=a1 and a3}] (intersection-1) node[circle, draw=white, fill=white, inner sep=1mm] {};
	\draw[name intersections={of=a1 and a2}] (intersection-1) node[regular polygon, regular polygon sides=3, rotate=180, draw=black, fill=white, inner sep=1mm] {};
	\draw (0, 1) .. controls (.5, 1) and (.5, 0) .. (1, 0) .. controls (1.5, 0) and (1.5, 1) .. (2, 1);
	\draw[name path=a1, draw=none] (3, 2) -- (5, 0);
	\draw[name path=a2, draw=none] (3, 0) -- (5, 2);
	\draw[name path=a3, draw=none] (3, 1) .. controls (3.5, 1) and (3.5, 2) .. (4, 2) .. controls (4.5, 2) and (4.5, 1) .. (5, 1);
	\draw (3, 2) -- (5, 0);
	\draw (3, 0) -- (5, 2);
	\draw[name intersections={of=a2 and a3}] (intersection-1) node[circle, draw=white, fill=white, inner sep=1mm] {};
	\draw[name intersections={of=a1 and a3}] (intersection-1) node[circle, draw=white, fill=white, inner sep=1mm] {};
	\draw[name intersections={of=a1 and a2}] (intersection-1) node[regular polygon, regular polygon sides=3, rotate=180, draw=black, fill=white, inner sep=1mm] {};
	\draw (3, 1) .. controls (3.5, 1) and (3.5, 2) .. (4, 2) .. controls (4.5, 2) and (4.5, 1) .. (5, 1);
	\draw[<->] (2.2, 1) -- (2.8, 1);
	\end{scope}
	\begin{scope}[cm={1,0,0,1,(0,-7.5)}]
	\draw[name path=a1, draw=none] (0, 2) -- (2, 0);
	\draw[name path=a2, draw=none] (0, 0) -- (2, 2);
	\draw[name path=a3, draw=none] (0, 1) .. controls (.5, 1) and (.5, 0) .. (1, 0) .. controls (1.5, 0) and (1.5, 1) .. (2, 1);
	\draw (0, 1) .. controls (.5, 1) and (.5, 0) .. (1, 0) .. controls (1.5, 0) and (1.5, 1) .. (2, 1);
	\draw[name intersections={of=a2 and a3}] (intersection-1) node[circle, draw=white, fill=white, inner sep=1mm] {};
	\draw[name intersections={of=a1 and a3}] (intersection-1) node[circle, draw=white, fill=white, inner sep=1mm] {};
	\draw (0, 2) -- (2, 0);
	\draw[name intersections={of=a1 and a2}] (intersection-1) node[regular polygon, regular polygon sides=3, draw=black, fill=black, inner sep=1mm] {};
	\draw (0, 0) -- (2, 2);
	\draw[name path=a1, draw=none] (3, 2) -- (5, 0);
	\draw[name path=a2, draw=none] (3, 0) -- (5, 2);
	\draw[name path=a3, draw=none] (3, 1) .. controls (3.5, 1) and (3.5, 2) .. (4, 2) .. controls (4.5, 2) and (4.5, 1) .. (5, 1);
	\draw (3, 1) .. controls (3.5, 1) and (3.5, 2) .. (4, 2) .. controls (4.5, 2) and (4.5, 1) .. (5, 1);
	\draw[name intersections={of=a2 and a3}] (intersection-1) node[circle, draw=white, fill=white, inner sep=1mm] {};
	\draw[name intersections={of=a1 and a3}] (intersection-1) node[circle, draw=white, fill=white, inner sep=1mm] {};
	\draw (3, 2) -- (5, 0);
	\draw[name intersections={of=a1 and a2}] (intersection-1) node[regular polygon, regular polygon sides=3, draw=black, fill=black, inner sep=1mm] {};
	\draw (3, 0) -- (5, 2);
	\draw[<->] (2.2, 1) -- (2.8, 1);
	\end{scope}
	\begin{scope}[cm={1,0,0,1,(0,-10)}]
	\draw[name path=a1, draw=none] (0, 2) -- (2, 0);
	\draw[name path=a2, draw=none] (0, 0) -- (2, 2);
	\draw[name path=a3, draw=none] (0, 1) .. controls (.5, 1) and (.5, 0) .. (1, 0) .. controls (1.5, 0) and (1.5, 1) .. (2, 1);
	\draw (0, 1) .. controls (.5, 1) and (.5, 0) .. (1, 0) .. controls (1.5, 0) and (1.5, 1) .. (2, 1);
	\draw (0, 2) -- (2, 0);
	\draw (0, 0) -- (2, 2);
	\draw[name intersections={of=a2 and a3}] (intersection-1) node[regular polygon, regular polygon sides=3, rotate=180, draw=black, fill=white, inner sep=1mm] {};
	\draw[name intersections={of=a1 and a3}] (intersection-1) node[regular polygon, regular polygon sides=3, rotate=180, draw=black, fill=white, inner sep=1mm] {};
	\draw[name intersections={of=a1 and a2}] (intersection-1) node[regular polygon, regular polygon sides=3, rotate=180, draw=black, fill=white, inner sep=1mm] {};
	\draw[name path=a1, draw=none] (3, 2) -- (5, 0);
	\draw[name path=a2, draw=none] (3, 0) -- (5, 2);
	\draw[name path=a3, draw=none] (3, 1) .. controls (3.5, 1) and (3.5, 2) .. (4, 2) .. controls (4.5, 2) and (4.5, 1) .. (5, 1);
	\draw (3, 1) .. controls (3.5, 1) and (3.5, 2) .. (4, 2) .. controls (4.5, 2) and (4.5, 1) .. (5, 1);
	\draw (3, 2) -- (5, 0);
	\draw (3, 0) -- (5, 2);
	\draw[name intersections={of=a2 and a3}] (intersection-1) node[regular polygon, regular polygon sides=3, rotate=180, draw=black, fill=white, inner sep=1mm] {};
	\draw[name intersections={of=a1 and a3}] (intersection-1) node[regular polygon, regular polygon sides=3, rotate=180, draw=black, fill=white, inner sep=1mm] {};
	\draw[name intersections={of=a1 and a2}] (intersection-1) node[regular polygon, regular polygon sides=3, rotate=180, draw=black, fill=white, inner sep=1mm] {};
	\draw[<->] (2.2, 1) -- (2.8, 1);
	\end{scope}
	\begin{scope}[cm={1,0,0,1,(0,-12.5)}]
	\draw[name path=a1, draw=none] (0, 2) -- (2, 0);
	\draw[name path=a2, draw=none] (0, 0) -- (2, 2);
	\draw[name path=a3, draw=none] (0, 1) .. controls (.5, 1) and (.5, 0) .. (1, 0) .. controls (1.5, 0) and (1.5, 1) .. (2, 1);
	\draw (0, 1) .. controls (.5, 1) and (.5, 0) .. (1, 0) .. controls (1.5, 0) and (1.5, 1) .. (2, 1);
	\draw (0, 2) -- (2, 0);
	\draw (0, 0) -- (2, 2);
	\draw[name intersections={of=a2 and a3}] (intersection-1) node[regular polygon, regular polygon sides=3, draw=black, fill=black, inner sep=1mm] {};
	\draw[name intersections={of=a1 and a3}] (intersection-1) node[regular polygon, regular polygon sides=3, draw=black, fill=black, inner sep=1mm] {};
	\draw[name intersections={of=a1 and a2}] (intersection-1) node[regular polygon, regular polygon sides=3, draw=black, fill=black, inner sep=1mm] {};
	\draw[name path=a1, draw=none] (3, 2) -- (5, 0);
	\draw[name path=a2, draw=none] (3, 0) -- (5, 2);
	\draw[name path=a3, draw=none] (3, 1) .. controls (3.5, 1) and (3.5, 2) .. (4, 2) .. controls (4.5, 2) and (4.5, 1) .. (5, 1);
	\draw (3, 1) .. controls (3.5, 1) and (3.5, 2) .. (4, 2) .. controls (4.5, 2) and (4.5, 1) .. (5, 1);
	\draw (3, 2) -- (5, 0);
	\draw (3, 0) -- (5, 2);
	\draw[name intersections={of=a2 and a3}] (intersection-1) node[regular polygon, regular polygon sides=3, draw=black, fill=black, inner sep=1mm] {};
	\draw[name intersections={of=a1 and a3}] (intersection-1) node[regular polygon, regular polygon sides=3, draw=black, fill=black, inner sep=1mm] {};
	\draw[name intersections={of=a1 and a2}] (intersection-1) node[regular polygon, regular polygon sides=3, draw=black, fill=black, inner sep=1mm] {};
	\draw[<->] (2.2, 1) -- (2.8, 1);
	\end{scope}
	\begin{scope}[cm={1,0,0,1,(0,-15)}]
	\draw[name path=a1, draw=none] (0, 2) -- (2, 0);
	\draw[name path=a2, draw=none] (0, 0) -- (2, 2);
	\draw[name path=a3, draw=none] (0, 1) .. controls (.5, 1) and (.5, 0) .. (1, 0) .. controls (1.5, 0) and (1.5, 1) .. (2, 1);
	\draw (0, 1) .. controls (.5, 1) and (.5, 0) .. (1, 0) .. controls (1.5, 0) and (1.5, 1) .. (2, 1);
	\draw (0, 2) -- (2, 0);
	\node[draw=white,fill=white,inner sep=1mm] at (1, 1) {};
	\draw (0, 0) -- (2, 2);
	\draw[name intersections={of=a2 and a3}] (intersection-1) node[regular polygon, regular polygon sides=3, draw=black, fill=black, inner sep=1mm] {};
	\draw[name intersections={of=a1 and a3}] (intersection-1) node[regular polygon, regular polygon sides=3, draw=black, fill=black, inner sep=1mm] {};
	\draw[name path=a1, draw=none] (3, 2) -- (5, 0);
	\draw[name path=a2, draw=none] (3, 0) -- (5, 2);
	\draw[name path=a3, draw=none] (3, 1) .. controls (3.5, 1) and (3.5, 2) .. (4, 2) .. controls (4.5, 2) and (4.5, 1) .. (5, 1);
	\draw (3, 1) .. controls (3.5, 1) and (3.5, 2) .. (4, 2) .. controls (4.5, 2) and (4.5, 1) .. (5, 1);
	\draw (3, 2) -- (5, 0);
	\node[draw=white,fill=white,inner sep=1mm] at (4, 1) {};
	\draw (3, 0) -- (5, 2);
	\draw[name intersections={of=a2 and a3}] (intersection-1) node[regular polygon, regular polygon sides=3, draw=black, fill=black, inner sep=1mm] {};
	\draw[name intersections={of=a1 and a3}] (intersection-1) node[regular polygon, regular polygon sides=3, draw=black, fill=black, inner sep=1mm] {};
	\draw[<->] (2.2, 1) -- (2.8, 1);
	\end{scope}
	\end{scope}

	\draw[dotted] (5.5, 2) -- (5.5, -15);
	\draw[dotted] (11.5, 2) -- (11.5, -15);
	\draw[dotted] (17.5, 2) -- (17.5, -15);
\end{tikzpicture}
\caption{Local moves}\label{F:reid}
\end{center}
\end{figure}

\subsection{Allowed moves}

Figure \ref{F:reid} shows the possible local changes. Which changes are allowed depends on the crossing type of the labeled crossings. Our general rule is to take all moves that involve crossing types possible in the diagrams, but we will also give an explicit list for each diagram type. We will refer to moves by the letter and number shown in Fig. \ref{F:reid}.

For classical links, local changes to crossings are given by the Reidemeister moves, shown as moves $a1$, $a2$, and $a3$. Details may be found in standard textbooks, such as Kauffman \cite{kauffman06} or Murasugi \cite{murasugi}. 

For virtual links, which contain classical and virtual crossings, we reference Kauffman \cite{vkt99}. Moves permitted for virtual links are $a1$, $e1$, $a2$, $e2$, $a3$, $a4$, and $b4$. For virtual links that contain rigid crossings, we add moves $b3$, $c3$, $d3$, and $g1$. For flat virtual links (also described in \cite{vkt99}), which contain flat and virtual crossings, we permit moves $b1$, $e1$, $b2$, $e2$, $e3$, $f3$, and $b4$.

For welded links, which contain classical and welded (down) crossings, we reference Fenn, Rimanyi, and Rourke \cite{welded97}. Moves permitted for welded links are $a1$, $c1$, $a2$, $c2$, $a3$, $g3$, $c4$, and $e4$. Welded knots differ from virtual knots by allowing an extra move -- $c4$ -- in which ordinary crossings are passed over a welded crossing. We create a distinction between welded (down) knots and the variety given by allowing ordinary crossings to pass under a welded (up) crossing. This alternate theory allows moves $a1$, $d1$, $a2$, $d2$, $a3$, $g4$, $d4$, and $f4$.

\section{An alternate construction}

We now give an alternate construction. We will show that this construction is equivalent, up to allowed moves, to the construction previously exhibited. We first give a formal specification of a language. This specification is given in EBNF, extended Backus-Naur form, defined in the ISO standard 14977:1996 \cite{isoebnf}. In EBNF, characters have the following meanings: the vertical bar \texttt{|} functions as an 'or' indicator, separating the elements of a set of which any one may be chosen; the characters \texttt{\{\}} enclose items repeated $0$ or more times; literals are enclosed in quotes; commas indicate concatenation; the \texttt{*} character indicates repetition; the \texttt{-} character indicates the set minus operation; and parentheses act as grouping.

\begin{definition}[tar sentence, virtual word, welded word, basic word, flat word set, rigid word set, circle word, label] \emph{Tar sentences} and the associated pieces \emph{virtual word}, \emph{welded word}, \emph{basic word}, \emph{flat word set}, \emph{rigid word set}, \emph{circle word}, and \emph{label} are in the space of words defined by the syntax in Fig. \ref{F:syntax}. Those sentences that match this syntax and can be drawn as described in the next subsection are tar sentences.
\end{definition}
\begin{figure}[ht]
\begin{verbatim}
tar sentence = {(virtual word | welded word)}, {(basic word |
  flat word set  | rigid word set)}, {welded word}, circle word;
virtual word = "[", basic word, "]";
welded word = "(", basic word, ")";
basic word = ("+" | "-"), 2 * label, {label};
flat word set = "{", basic word, {basic word}, "}";
rigid word set = "!", 2 * basic word;
circle word = ";", 2 * label, {label};
label = digit - "0", {digit}, ".";
digit = "0" | "1" | "2" | "3" | "4" | "5" | "6" | "7" | "8" | "9";
\end{verbatim}
\caption{Formal syntax for tar sentences}\label{F:syntax}
\end{figure}

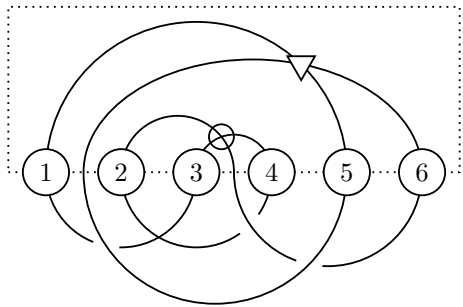
\begin{figure}[ht]
\begin{center}
\begin{tikzpicture}[line width=.7pt]
	\draw[draw=none, name path=a1] (2, 0) arc[start angle=180, end angle=0, radius=.5];
	\draw[draw=none, name path=a2] (0, 0) arc[start angle=180, end angle=0, radius=2];
	\draw[draw=none, name path=a3] (5, 0) arc[start angle=0, end angle=180, y radius=1.5, x radius=2.25] arc[start angle=-180, end angle=0, radius=1.75];
	\draw[draw=none, name path=a4] (1, 0) arc[start angle=180, end angle=0, radius=.75] arc[start angle=-180, end angle=0, radius=1.25];
	\draw[draw=none, name path=a5] (0, 0) arc[start angle=-180, end angle=0, radius=1];
	\draw[draw=none, name path=a6] (1, 0) arc[start angle=-180, end angle=0, radius=1];

	\draw (1, 0) arc[start angle=-180, end angle=0, radius=1];
	\draw[name intersections={of=a4 and a6}] (intersection-2) node[circle,draw=white,fill=white] {};
	\draw (0, 0) arc[start angle=-180, end angle=0, radius=1];
	\draw[name intersections={of=a3 and a5}] (intersection-1) node[circle,draw=white,fill=white] {};
	\draw (1, 0) arc[start angle=180, end angle=0, radius=.75] arc[start angle=-180, end angle=0, radius=1.25];
	\draw[name intersections={of=a3 and a4}] (intersection-2) node[circle,draw=white,fill=white] {};
	\draw[name intersections={of=a4 and a1}] (intersection-1) node[circle,draw=black,fill=none] {};
	\draw (5, 0) arc[start angle=0, end angle=180, y radius=1.5, x radius=2.25] arc[start angle=-180, end angle=0, radius=1.75];
	\draw (0, 0) arc[start angle=180, end angle=0, radius=2];
	\draw (2, 0) arc[start angle=180, end angle=0, radius=.5];
	\draw[name intersections={of=a2 and a3}] (intersection-1) node[scale=.75, regular polygon, rotate=180, regular polygon sides=3, draw=black, fill=white, inner sep=1mm] {};
	
	\draw[dotted] (-.5, 0) -- (5.5, 0) -- (5.5, 2.2) -- (-.5, 2.2) -- cycle;
	\begin{scope}[every node/.style={circle,draw=black,fill=white}]
	\node at (1, 0) {$2$};
	\node at (5, 0) {$6$};
	\end{scope}
	\begin{scope}[every node/.style={circle,draw=black,fill=white}]
	\node at (0, 0) {$1$};
	\node at (4, 0) {$5$};
	\node at (2, 0) {$3$};
	\node at (3, 0) {$4$};
	\end{scope}
	
	\node at (2.5, -2.5) {\texttt{[+3.4.](+5.1.)+6.1.5.+2.3.6.\{-1.3.-4.2.\};1.2.3.4.5.6.}};
\end{tikzpicture}
\caption{A tar sentence with various crossing types and an associated diagram; the link orientation is from \texttt{1} toward \texttt{3}.}\label{F:exkishino}
\end{center}
\end{figure}

Let us give an example. Figure \ref{F:exkishino} may be represented by the tar sentence \texttt{[+3.4.](+5.1.)+6.1.5.+2.3.6.\{-1.3.-4.2.\};1.2.3.4.5.6.} as demonstrated in future sections; here, we are interested in the component pieces of the syntax. The basic letters of the language are integers -- which we separate by the use of the \texttt{.} character -- and the punctuation marks \texttt{\{ \} ( ) + - ;} and \texttt{!}. In the tar sentence we have chosen, our letters are the integers $1$ through $6$ and a subset of the puncutation. The basic words are \texttt{+3.4.}, \texttt{+5.1.}, \texttt{+6.1.5.}, \texttt{+2.3.6.}, \texttt{-1.3.}, and \texttt{-4.2.}. Additionally, we have the circle word \texttt{;1.2.3.4.5.6.}.

We note that the number of basic words is equal to the number of integers used. This will be true of all tar sentences because each basic word (in conjunction with the circle word of the sentence) describes, up to an end-point fixing homotopy, an embedding of a path into a $2$-sphere with holes. As well, the set of integers that are at the start of a basic word will be the complete set of integers used; the same is true for integers that are at the end of a basic word. No basic word will begin and end with the same integer. The first two basic words are enclosed like this: \texttt{[+3.4.]} and \texttt{(+5.1.)}. This indicates that the two basic words are of virtual and welded \emph{types}, respectively. 

The rules for turning word types into crossing types are given in the next subsection, but determining the type of the word can be done directly from the syntax: virtual words are enclosed in \texttt{[]}; welded words are enclosed in \texttt{()} -- whether it is welded up or down depends on the placement of the welded word in the sentence; all of the basic words in the same flat hierarchy are enclosed in \texttt{\{\}}; rigid words (which always come in pairs) are preceded by a \texttt{!}; and every basic word with no other type is classical. The shape of the path described by a word depends only on the basic word. Looking at figure \ref{F:exkishino} again, we see that the last two basic words are enclosed in braces; the crossing between the pieces of the diagram represented by these two words -- the arc from \texttt{1} to \texttt{3} and the arc from \texttt{4} to \texttt{2} is flat. The only type of crossing that has a notion of upper and lower is the classical crossing. For these, which is upper and which is lower is determined by the order of the words: upper words precede lower words.

The circle word is a permutation of the integers used to form the basic words; it may be identified in the tar sentence because it begins with a \texttt{;} character and is always the final word of the sentence. We see that in the diagram the labels are listed in the circle word in the same order that they are encountered when traveling around the dotted line. We will see how this gives shape to a diagram and how it may be chosen from a diagram in future subsections. In essence, the circle word gives the location of the holes in the $2$-spheres used by the basic words.

Of special note is that a link diagram may be represented by a multitude of distinct tar sentences. As a result, many of the changes listed as moves for the tar sentences may not alter the diagram.

\subsection{How to draw a diagram from a tar sentence}

A link diagram in $S^2$ may be constructed from a tar sentence $L$ by the following steps. The crossing types of a diagram are determined by the rules given after the construction process.

\begin{enumerate}
\item Let $V_n = v_1, v_2, \dots, v_n$ be the sequence of labels in the circle word of $L$.
\item Let $c(v_i)$ be the number of occurrences of the label $v_i$ in the basic words of $L$.
\item If the elements of $V$ are not all distinct or the number of distinct labels in the basic words of $L$ is not $n$, then $L$ is not valid.
\item Let $m_1 = m_{n+1} = 0, m_j = \sum_{i=1}^{j-1} c(v_i) - 1$.
\item Let $k = m_n + c(v_n) - 1$.
\item Draw the labels $v_1, v_2, \dots, v_n$ at angles $\frac{2\pi m_1}{k}, \frac{2\pi m_2}{k}, \dots, \frac{2\pi m_n}{k}$, respectively, on the unit circle of $\mathbb{R}^2$.
\item For each basic word $w$ of $L$, we construct a curve on $S^2 \in S^3$ in the following manner.
\begin{enumerate}
\item If $w$ begins with \texttt{+}, the first arc goes through the upper hemisphere of $S^2 \in S^3$. Otherwise, the first arc goes through the lower hemisphere of $S^2 \in S^3$.
\item Arcs alternate between hemispheres and project (by $(x, y, z) \mapsto (x, y)$) to $\mathbb{R}^2$ as chords of the unit circle.
\item The first arc begins at the label matching the first label of $w$; the last arc ends at the label matching the last label of $w$.
\item Each pair of consecutive labels $v_iv_j$ in $w$ projects to a chord with ends in the arcs $\left[\frac{2\pi m_i}{k}, \frac{2\pi m_{i+1}}{k}\right)$ and $\left[\frac{2\pi m_j}{k}, \frac{2\pi m_{j+1}}{k}\right)$ of the unit circle.
\end{enumerate}
\item The ends of the chords of the previous step are chosen to be evenly spaced around the unit circle and to minimize the number of crossings.
\item If any basic word cannot be drawn to avoid self-crossings of the curve, $L$ is not valid.
\end{enumerate}

Figure \ref{F:basic_examples} shows some basic words with projections of their diagrams. It is interesting to notice, assuming a circle word containing \texttt{1.2.3.4.}, that \texttt{+1.2.4.} and \texttt{+1.2.4.2.4.} are valid basic words but \texttt{+1.2.4.2.} is not.

\begin{figure}[ht]
\begin{center}
\begin{tikzpicture}[line width=.7pt]

	\begin{scope}[cm={1,0,0,1,(0,0)}]
	\draw[dotted] (-.5, 0) -- (3.5, 0);
	\node[draw=black, circle, fill=white] (v1) at (0, 0) {$1$};
	\node[draw=black, circle, fill=white] (v2) at (1, 0) {$2$};
	\node[draw=black, circle, fill=white] (v3) at (2, 0) {$3$};
	\node[draw=black, circle, fill=white] (v4) at (3, 0) {$4$};
	\draw[->] (v1) .. controls (0, 1) and (2, 1) .. (v3);
	\node at (1.5, -.7) {\texttt{+1.3.}};
	\end{scope}

	\begin{scope}[cm={1,0,0,1,(6,0)}]
	\draw[dotted] (-.5, 0) -- (3.5, 0);
	\node[draw=black, circle, fill=white] (v1) at (0, 0) {$1$};
	\node[draw=black, circle, fill=white] (v2) at (1, 0) {$2$};
	\node[draw=black, circle, fill=white] (v3) at (2, 0) {$3$};
	\node[draw=black, circle, fill=white] (v4) at (3, 0) {$4$};
	\draw[<-] (v1) .. controls (0, 1) and (2, 1) .. (v3);
	\node at (1.5, -.7) {\texttt{+3.1.}};
	\end{scope}

	\begin{scope}[cm={1,0,0,1,(0,-2)}]
	\draw[dotted] (-.5, 0) -- (3.5, 0);
	\node[draw=black, circle, fill=white] (v1) at (0, 0) {$1$};
	\node[draw=black, circle, fill=white] (v2) at (1, 0) {$2$};
	\node[draw=black, circle, fill=white] (v3) at (2, 0) {$3$};
	\node[draw=black, circle, fill=white] (v4) at (3, 0) {$4$};
	\draw[->] (v1) .. controls (0, -1) and (2, -1) .. (v3);
	\node at (1.5, .7) {\texttt{-1.3.}};
	\end{scope}

	\begin{scope}[cm={1,0,0,1,(6,-2)}]
	\draw[dotted] (-.5, 0) -- (3.5, 0);
	\node[draw=black, circle, fill=white] (v1) at (0, 0) {$1$};
	\node[draw=black, circle, fill=white] (v2) at (1, 0) {$2$};
	\node[draw=black, circle, fill=white] (v3) at (2, 0) {$3$};
	\node[draw=black, circle, fill=white] (v4) at (3, 0) {$4$};
	\draw[->] (v1) .. controls (0, 1) and (1.5, 1) .. (1.5, 0) .. controls (1.5, -1) and (3.5, -1) .. (3.5, 0) .. controls (3.5, 1) and (2, 1) .. (v3);
	\node at (1, -.7) {\texttt{+1.2.4.3.}};
	\end{scope}

	\begin{scope}[cm={1,0,0,1,(0,-4)}]
	\draw[dotted] (-.5, 0) -- (3.5, 0);
	\node[draw=black, circle, fill=white] (v1) at (0, 0) {$1$};
	\node[draw=black, circle, fill=white] (v2) at (1, 0) {$2$};
	\node[draw=black, circle, fill=white] (v3) at (2, 0) {$3$};
	\node[draw=black, circle, fill=white] (v4) at (3, 0) {$4$};
	\draw[->] (v1) .. controls (0, .5) and (.5, .5) .. (.5, 0) .. controls (.5, -1) and (2, -1) .. (v3);
	\node at (1.5, .7) {\texttt{+1.1.3.}};
	\end{scope}

	\begin{scope}[cm={1,0,0,1,(6,-4)}]
	\draw[dotted] (-.5, 0) -- (3.5, 0);
	\node[draw=black, circle, fill=white] (v1) at (0, 0) {$1$};
	\node[draw=black, circle, fill=white] (v2) at (1, 0) {$2$};
	\node[draw=black, circle, fill=white] (v3) at (2, 0) {$3$};
	\node[draw=black, circle, fill=white] (v4) at (3, 0) {$4$};
	\draw[->] (v1) .. controls (0, 1) and (1.35, 1) .. (1.35, 0) .. controls (1.3, -1) and (3.5, -1) .. (3.5, 0) .. 
		controls (3.5, 1) and (1.55, 1) .. (1.55, 0) .. controls (1.65, -.75) and (3, -.75) .. (v4);
	\node at (.8, -.7) {\texttt{+1.2.4.2.4.}};
	\end{scope}

\end{tikzpicture}
\caption{Some basic words and associated diagrams. All have a circle word containing \texttt{1.2.3.4.}.}\label{F:basic_examples}
\end{center}
\end{figure}
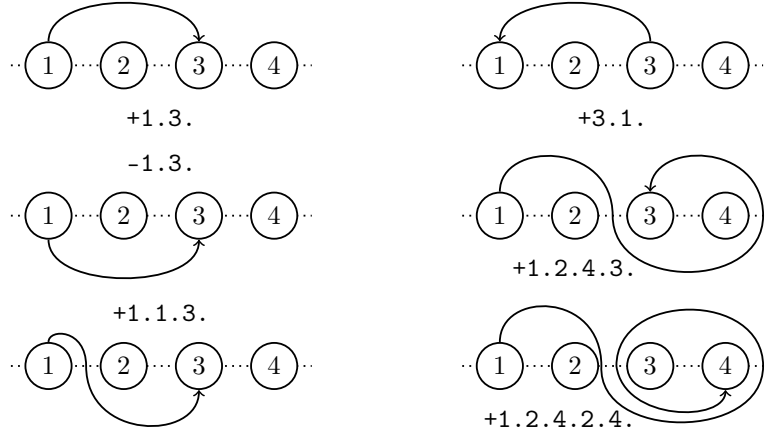

The following rules determine the type of every crossing between two curves drawn from words $w_i$ and $w_j$ in $L = \cdots w_i \cdots w_j \cdots$.

\begin{enumerate}
\item If $w_i$ or $w_j$ is a virtual word, then the crossings are virtual; otherwise,
\item if $w_i$ or $w_j$ is a welded word and the other is not or there is no word $w_k$ between $w_i$ and $w_j$ with $w_k$ non-welded and non-virtual, then the crossings are welded -- up for a welded word at the beginning of the sentence, down for a welded word in the group adjacent to the circle word; otherwise,
\item if $w_i$ and $w_j$ are welded, then the link is non-standard and the crossings are also non-standard; otherwise,
\item if $w_i$ and $w_j$ are in the same rigid word set, then the crossing is rigid; otherwise,
\item if $w_i$ and $w_j$ are in the same flat word set, then the crossings are flat; otherwise,
\item the crossing is classical, and $w_i$ crosses over $w_j$.
\end{enumerate}

We will give an interpretation of a crossing between welded (up) and welded (down) trails in the section on moves of a tar sentence.

\subsection{How to construct a tar sentence from a link diagram}

Let $L$ be a link diagram. To $L$ we add some nodes by subdividing arcs of $L$. We assign each new node, all of which are of degree $2$, the type \emph{riser mark}. Let $L'$ be a link diagram created in this way by subdividing each arc of $L$ twice.

We now consider paths of length $2$ where the center node is of degree $4$ and the path forms a transverse crossing with the path comprising the other $2$ paths of the center node. To each of these paths we assign the same type as the type of the center node. This leaves paths of length $1$ whose ends are both nodes of degree $2$. To each of these paths we assign the classical type.

\begin{definition}[fully partitioned link diagram] A projection to $\mathbb{R}^2$ of a link diagram in which all arcs have been twice subdivided and types have been assigned such that every arc is assigned a type and no arc has a type different than its head or tail node, whichever is a crossing, is a \emph{fully partitioned link diagram}.
\end{definition}

\begin{definition}[riser mark, trail]A \emph{riser mark} is a degree two vertex. A \emph{trail} is a path that begins and ends at a riser mark and crosses each degree $4$-vertex transversely.
\end{definition}

\begin{definition}[binding circle, inside, outside]A \emph{binding circle} of a fully partitioned link diagram is an oriented, self-avoiding loop which contains each riser mark of the diagram and crosses all other points of the diagram transversely.  The region bordering the binding circle is \emph{outside} the binding circle, and the finite region bounded by the binding circle is \emph{inside} the binding circle. The binding circle is oriented such that the winding number of any point inside the binding circle is one.
\end{definition}

We label the riser marks of $L$ with distinct, positive integers and construct the basic words of $L$ in the following way.

\begin{enumerate}
\item{If there is a neighborhood of the tail of the trail distinct from the outside of the binding circle, then the first letter is \texttt{+}.}
\item{If there is a neighborhood of the tail of the trail distinct from the inside of the binding circle, then the first letter is \texttt{-}.}
\item{The second and third letters are the label of the riser mark at the tail of the trail and the \texttt{.} character respectively.}
\item{At each crossing of the trail and the binding circle, taking crossings sequentially along the trail, the label of the riser mark first encountered when following the binding circle counter to its orientation from the crossing is appended, followed by the \texttt{.} character.}
\item{The last two letters of the basic word are the label of the riser mark at the head of the trail and the \texttt{.} character.} 
\end{enumerate}

From the basic word $w$ of each trail, we construct the tar sentence in the following order:
\begin{enumerate}
\item The virtual trails are written in any order by writing \texttt{[} $w$ \texttt{]} for each.
\item The welded (up) trails are written in any order by writing \texttt{(} $w$ \texttt{)} for each.
\item The classical, rigid, and flat trails are written in an order determined by their classical crossings. If $a$ crosses over $b$, $a$ must appear before $b$.
\item Flat trails $w_1$ and $w_2$ that share a crossing are written \texttt{\{} $w_1$ $w_2$ \texttt{\}}.
\item Rigid trails $w_1$ and $w_2$ that share a crossing are written \texttt{!} $w_1$ $w_2$.
\item The welded (down) trails are written in any order by writing \texttt{(} $w$ \texttt{)} for each.
\item The binding circle is written as \texttt{;} $l_1$ \texttt{.} $l_2$ \texttt{.} $\cdots$ where $l_1$ is the label of an arbitrarily chosen riser mark and each $l_i$ is the label of the next riser mark encountered when following the binding circle in its orientation.
\end{enumerate}

\begin{proposition}The tar sentence of a link diagram created as described allows reconstruction of the diagram up to isotopy.
\end{proposition}

\begin{proof}A link diagram with the riser marks and binding circle described is partitioned into disjoint crossings and arcs by the removal of the binding circle. The tar sentence describes a map for stitching together the pieces of the diagram.
\end{proof}

\section{Moves on the alternate construction}

The construction of a tar sentence given produces a representation of a link that is, in most cases, longer than a representation that relies on enumerating the same set of crossings. We proceed now to show two reasons why it is of interest. First, the moves on tar sentences are more global in nature than the local crossing changes of the standard representations. Second, the tar sentence may be compacted to a form whose length depends only on the number of riser marks. Because the minimal number of riser marks is twice the bridge index for classical links and comparable for many non-classical links (specifically excluding rigid vertices), links with very high crossing numbers can often be expressed in a brief way. In this section, we describe the moves on tar sentences.

Given a tar sentence, we can make moves of the following types.
\begin{enumerate}
\item riser removal (def. \ref{D:riserrem})
\item trail swap (def. \ref{D:trailswap})
\item duplicate removal (def. \ref{D:duplicateremoval})
\item riser addition (def. \ref{D:riseradd})
\item trail type conversion (def. \ref{D:trailtype})
\item riser move (def. \ref{D:risermove})
\item trail move (def. \ref{D:trailmove})
\item rigid flip (def. \ref{D:rigidflip})
\end{enumerate}

\begin{definition}[riser removal]\label{D:riserrem}A \emph{riser removal} of a sentence is the concatenation of two basic words. In order to be concatenated, the two basic words must meet the following requirements:
\begin{enumerate}
\item neither of the words may be in a rigid word set
\item when the punctuation, excepting \texttt{+ -}, is removed, the basic words must be consecutive
\item the final label of the first word must be the first label of the second
\item the trails represented by the basic words must not cross
\item if either of the words is in a flat word set, the other word must not intersect with any trails represented by members of that set
\item if the words are of different types, then each must be either basic or flat
\item there are at least three basic words in the link component of the diagram of which the words are a partial representation
\end{enumerate}
For basic words $a$ and $b$ which meet these criteria, the concatenation is the result of the following process.
\begin{enumerate}
\item If the number of labels in the first word is even, then:
	\begin{enumerate}
	\item if the first letters of $a$ and $b$ are the same, then remove the final label of $a$ and the first letter and first label of $b$
	\item if the first letters of $a$ and $b$ differ, remove the final label of $a$ and the first letter of $b$
	\end{enumerate}
\item If the number of labels in the first word is odd, then:
	\begin{enumerate}
	\item if the first letters of $a$ and $b$ are the same, then remove the final label of $a$ and the first letter of $b$
	\item if the first letters of $a$ and $b$ differ, remove the final label of $a$ and the first letter and first label of $b$
	\end{enumerate}
\item In all basic words replace all occurrences of the last label $l$ of $a$ with the label immediately before $l$ in the circle word of the sentence; if $l$ is the first label of the circle word, replace $l$ with the last label of the circle word.
\item Replace the punctuation indicating word type (e.g. virtual word, flat word set, etc.) according to the following rules:
	\begin{enumerate}
	\item words of the same type keep that type
	\item a basic word and a flat word form a word that belongs in the same flat word set as the original flat word
	\item if both words are flat and from different flat word sets, the two sets are combined
	\end{enumerate}
\end{enumerate}

Figure \ref{F:riser_rem} shows a partial diagram with riser removal.

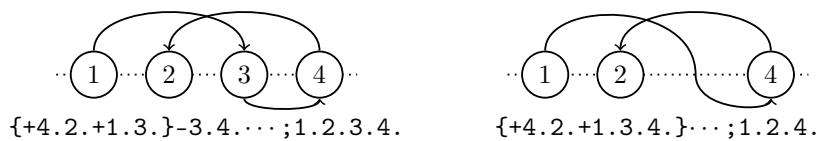
\begin{figure}[ht]
\begin{center}
\begin{tikzpicture}[line width=.7pt]

	\begin{scope}[cm={1,0,0,1,(0,0)}]
	\draw[dotted] (-.5, 0) -- (3.5, 0);
	\node[draw=black, circle, fill=white] (v1) at (0, 0) {$1$};
	\node[draw=black, circle, fill=white] (v2) at (1, 0) {$2$};
	\node[draw=black, circle, fill=white] (v3) at (2, 0) {$3$};
	\node[draw=black, circle, fill=white] (v4) at (3, 0) {$4$};
	\draw[->] (v1) .. controls (0, 1) and (2, 1) .. (v3);
	\draw[->] (v3) .. controls (2, -.5) and (3, -.5) .. (v4);
	\draw[<-] (v2) .. controls (1, 1) and (3, 1) .. (v4);
	\node at (1.5, -.7) {\texttt{\{+4.2.+1.3.\}-3.4.}$\cdots$\texttt{;1.2.3.4.}};
	\end{scope}

	\begin{scope}[cm={1,0,0,1,(6,0)}]
	\draw[dotted] (-.5, 0) -- (3.5, 0);
	\node[draw=black, circle, fill=white] (v1) at (0, 0) {$1$};
	\node[draw=black, circle, fill=white] (v2) at (1, 0) {$2$};
	\node[draw=black, circle, fill=white] (v4) at (3, 0) {$4$};
	\draw[->] (v1) .. controls (0, 1) and (2, 1) .. (2, 0) .. controls (2, -.5) and (3, -.5) .. (v4);
	\draw[<-] (v2) .. controls (1, 1) and (3, 1) .. (v4);
	\node at (1.5, -.7) {\texttt{\{+4.2.+1.3.4.\}}$\cdots$\texttt{;1.2.4.}};
	\end{scope}
\end{tikzpicture}
\caption{A riser removal}\label{F:riser_rem}
\end{center}
\end{figure}
\end{definition}

\begin{definition}[trail swap]\label{D:trailswap}A \emph{trail swap} is the exchange of the position of two consecutive trail elements of a sentence. Trail elements are virtual words, welded words, rigid word sets, flat word sets, and basic words not in a virtual word or welded word. This exchange is permitted in the following cases. In all other cases, exchange is forbidden:
\begin{enumerate}
\item{both are members of the same flat word set}
\item{both are basic words and the trails represented by the two do not cross}
\item{both are virtual}
\item{both are welded}
\item{one is virtual and the other is welded}
\item{each is either a basic word, flat word set, or rigid word set and there are no crossings between trails represented by the first and the second.}
\end{enumerate}

A trail swap has induces no diagrammatic change. As an example, the left-hand side of fig. \ref{F:riser_rem} may be re-written as \texttt{\{+1.3.+4.2.\}-3.4.}$\cdots$\texttt{;1.2.3.4.} without altering the diagram.
\end{definition}

\begin{definition}[duplicate removal]\label{D:duplicateremoval}A \emph{duplicate removal} is the replacement of a pair of adjacent, identical labels according to the following rules:
\begin{enumerate}
\item if the duplication occurs in a rigid word set, then no change occurs; otherwise,
\item if the letter before the duplication is a \texttt{+}, then one of the duplicates is removed and the \texttt{+} is replaced by a \texttt{-}; otherwise,
\item if the letter before the duplication is a \texttt{-}, then one of the duplicates is removed and the \texttt{+} is replaced by a \texttt{+}; otherwise,
\item if the repeated labels are the final two labels in a word, then one of them is removed; otherwise,
\item if none of the previous cases apply, then both of the duplicates are removed
\end{enumerate}

Figure \ref{F:dup_removal} shows a duplicate removal. The change from \texttt{+1.1.3.} to \texttt{-1.3.} in fig. \ref{F:basic_examples} is also an example of a duplicate removal. 

\begin{figure}[ht]
\begin{center}
\begin{tikzpicture}[line width=.7pt]

	\begin{scope}[cm={1,0,0,1,(0,0)}]
	\draw[dotted] (-.5, 0) -- (3.5, 0);
	\node[draw=black, circle, fill=white] (v1) at (0, 0) {$1$};
	\node[draw=black, circle, fill=white] (v2) at (1, 0) {$2$};
	\node[draw=black, circle, fill=white] (v3) at (2, 0) {$3$};
	\node[draw=black, circle, fill=white] (v4) at (3, 0) {$4$};
	\draw[->] (v1) .. controls (0, 1) and (2.2, 1) .. (2.45, 0) .. controls (2.45, -.5) and (2.65, -.5) .. (2.65, 0) .. controls (2.65, .75) and (3, .75) .. (v4);
	\node at (1.5, -.7) {\texttt{+1.3.3.4.}$\cdots$\texttt{;1.2.3.4.}};
	\end{scope}

	\begin{scope}[cm={1,0,0,1,(6,0)}]
	\draw[dotted] (-.5, 0) -- (3.5, 0);
	\node[draw=black, circle, fill=white] (v1) at (0, 0) {$1$};
	\node[draw=black, circle, fill=white] (v2) at (1, 0) {$2$};
	\node[draw=black, circle, fill=white] (v3) at (2, 0) {$3$};
	\node[draw=black, circle, fill=white] (v4) at (3, 0) {$4$};
	\draw[->] (v1) .. controls (0, 1) and (3, 1) .. (v4);
	\node at (1.5, -.7) {\texttt{+1.4.}$\cdots$\texttt{;1.2.3.4.}};
	\end{scope}
\end{tikzpicture}
\caption{A duplicate removal}\label{F:dup_removal}
\end{center}
\end{figure}
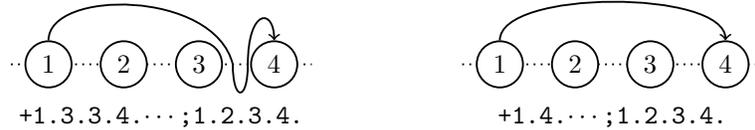

\end{definition}

\begin{definition}[riser addition]\label{D:riseradd}A \emph{riser addition} is the replacement of a basic word, not part of a rigid word set, by a pair of basic words. It occurs by the following procedure.
\begin{enumerate}
\item{We choose a basic word.}
\item{We let the final label of the word be $l$.}
\item{A new label, distinct from any label in the sentence, is chosen. We call this new label $t$.}
\item{In the circle word, we replace $l$ by the sequence $tl$.}
\item{We replace the final label of the basic word by $t$. We insert the word \texttt{+} $lt$ just after the chosen word with the same word type as the chosen word. If the chosen word is flat, the new word is placed in the same flat word set.}
\end{enumerate}

Figure \ref{F:riser_rem} shows a riser addition if the diagram is read from right to left.
\end{definition}

\begin{definition}[trail type conversion]\label{D:trailtype}A \emph{trail type conversion} is the replacement of one trail element by a different trail element containing the same basic words but of a different type. This may be done according to the following rules.
\begin{enumerate}
\item A virtual (welded) word for which the trail represented by the word has no crossings with any non-virtual (non-welded) trails and that follows all other virtual words and welded words representing welded (up) trails may be replaced by a basic word.
\item A flat word set containing only one basic word may be replaced by that basic word.
\item A basic word representing a flat trail that has no crossings with any of the trails represented by other members of its flat word set may be moved to immediately precede the flat word set. If this leaves an empty set, the previous rule is used instead.
\item A basic word with no classical crossings and which is preceded only by virtual and welded words may be replaced by a virtual or welded word.
\item A basic word with no classical crossings and which is followed only by welded words may be replaced by a welded word.
\end{enumerate}

The left-hand side of fig. \ref{F:riser_rem} may be written as \texttt{\{+4.2.+1.3.-3.4.\}}$\cdots$\texttt{;1.2.3.4.} instead of \texttt{\{+4.2.+1.3.\}-3.4.}$\cdots$\texttt{;1.2.3.4.}. This is a trail type conversion, and, like all such conversions, does not affect the diagram.
\end{definition}

For the remaining two moves, we need to define what we mean by a \emph{level diagram} and isotopies of these diagrams.

\begin{definition}[level diagram, standard level diagram set]A \emph{level diagram} of a tar sentence is a partial reconstruction of a link diagram from the tar sentence by the following method.
\begin{enumerate}
\item A trail element of the sentence or a basic word of a flat word set is chosen.
\item The binding circle is drawn and labeled according to the circle word.
\item Each basic word in the element is drawn.
\item Labels are removed from the binding circle according to the rules:
	\begin{enumerate}
	\item If the trail element is virtual or welded, all labels except those at the ends of the trail are removed, otherwise
	\item a label is removed if it does not belong to a virtual word and the basic words that contain this label as their first or last label both precede or follow the chosen element.
	\end{enumerate}
\end{enumerate}
We let the \emph{standard level diagram set} be the set of level diagrams given by taking the level set of each trail element except for flat word sets and taking the level set diagrams of each basic word in a flat word set.

Figures \ref{F:kishino} and \ref{F:kishinoproof} show a link diagram and the corresponding level diagrams.
\end{definition}

\begin{definition}[ambient isotopy]An \emph{ambient isotopy} is a continuous mapping $f: S^2 \times [0, 1] \to S^2$ such that every $f_t: S^2 \to S^2$ given by $f_t(X) = f(X, t)$ is a homeomorphism of $S^2$.
\end{definition}

\begin{definition}[level isotopy]A \emph{level isotopy} is an isotopy of a level diagram $L$ that is determined by an ambient isotopy $f$ of $S^2$ such that $f(L, 0) = L$.
\end{definition}

\begin{definition}[riser isotopy]A \emph{riser isotopy} of a tar sentence is a set of level isotopies $\{ f^1, f^2, \dots, f^n \}$, one for each corresponding level diagram $d_1, d_2, \dots, d_n$ in the standard level diagram set, such that all riser marks except one, which we assume is at position $r$, are constant under $f^i$ if they are drawn in $d_i$ and for every $i, j$ the equality $f^i(r, t)  = f^j(r, t)$ for all $t \in [0, 1]$ whenever $d_i$ and $d_j$ both draw the riser label at $r$. 
\end{definition}

\begin{definition}[riser move]\label{D:risermove}A \emph{riser move} is the replacement of all basic words by the basic words of the diagram that results from choosing a riser isotopy $f^1, f^2, \dots, f^n$ on levels $d_1, d_2, \dots, d^n$ and replacing each $f^i(d_i, 0)$ by $f^i(d_i, 1)$.
\end{definition}

Figure \ref{F:riser_move} shows part of a riser move. The trail represented by the basic word \texttt{+1.3.} is altered to the trail represented by \texttt{+1.2.3.} to accommodate the new location of riser mark \texttt{2}, which has moved in the manner shown by the dashed line. The circle word must, likewise, by modified to contain \texttt{1.3.2.4.} instead of \texttt{1.2.3.4.}.

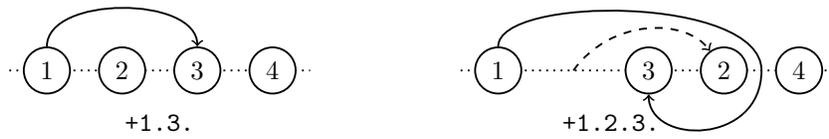
\begin{figure}[ht]
\begin{center}
\begin{tikzpicture}[line width=.7pt]

	\begin{scope}[cm={1,0,0,1,(0,0)}]
	\draw[dotted] (-.5, 0) -- (3.5, 0);
	\node[draw=black, circle, fill=white] (v1) at (0, 0) {$1$};
	\node[draw=black, circle, fill=white] (v2) at (1, 0) {$2$};
	\node[draw=black, circle, fill=white] (v3) at (2, 0) {$3$};
	\node[draw=black, circle, fill=white] (v4) at (3, 0) {$4$};
	\draw[->] (v1) .. controls (0, 1) and (2, 1) .. (v3);
	\node at (1.5, -.7) {\texttt{+1.3.}};
	\end{scope}

	\begin{scope}[cm={1,0,0,1,(6,0)}]
	\draw[dotted] (-.5, 0) -- (4.5, 0);
	\node[draw=black, circle, fill=white] (v1) at (0, 0) {$1$};
	\node[draw=black, circle, fill=white] (v2) at (3, 0) {$2$};
	\node[draw=black, circle, fill=white] (v3) at (2, 0) {$3$};
	\node[draw=black, circle, fill=white] (v4) at (4, 0) {$4$};
	\draw[->] (v1) .. controls (0, 1) and (3.5, 1) .. (3.5, 0) .. controls (3.5, -1) and (2, -1) .. (v3);
	\draw[->, dashed] (1, 0) .. controls (1.5, .7) and (2.5, .7) .. (v2);
	\node at (1.5, -.7) {\texttt{+1.2.3.}};
	\end{scope}

\end{tikzpicture}
\caption{A riser move by riser \texttt{2}.}\label{F:riser_move}
\end{center}
\end{figure}

\begin{definition}[trail move]\label{D:trailmove}A \emph{trail move} is the replacement of a basic word by the word describing the result of a level isotopy that keeps all labels appearing on the level diagram constant.
\end{definition}

\begin{definition}[rigid flip]\label{D:rigidflip}A \emph{rigid flip} is the replacement of a rigid word set \texttt{!}$w_1w_2$ by the rigid word set \texttt{!}$w_2'w_1'$ where $w_1'$ and $w_2'$ are the result of changing all \texttt{+} to \texttt{-} and all \texttt{-} to \texttt{+} in $w_1$ and $w_2$ respectively. A rigid flip is allowed if and only if the sentences obtained by replacing \texttt{!}$w_1w_2$ with $w_1w_2$ and with $w_2'w_1'$ are equivalent.
\end{definition}

\begin{theorem}For tar sentences $L_1$ and $L_2$, there is a sequence of tar sentence moves that transforms $L_1$ into $L_2$ if and only if the links represented by $L_1$ and $L_2$ are equivalent.
\end{theorem}

\begin{proof}We proceed by demonstrating that every move of Fig. \ref{F:reid} can be performed by tar sentence moves and that none of the forbidden moves of Fig. \ref{F:forbid} can be performed by tar sentence moves.

For moves $a1$, $b1$, $c1$, $d1$, and $e1$, we see that the right-hand side may be a basic word \texttt{+1.2.} with \texttt{1.2.} appearing in the circle word by a suitable choice of labels, riser additions, and riser removals. By trail swaps, this may be at whatever position in the sentence is necessary for the appropriate trail type conversion. The left-hand side may then be written as \texttt{+1.3.-3.4.+4.2.} with \texttt{1.3.4.2.} appearing in the circle word by two riser additions and a riser move. To move in the reverse direction, we reverse these moves, performing a riser move and then two riser removals.

For moves $a2$, $b2$, $c2$, $d2$, and $e2$, we assume that the arcs on the right-hand side are each a basic word. If not, we may perform riser additions and removals to make them so. By a riser addition and a riser move, we can create the left-hand side. To move from the right-hand side to the left-hand side, we note that riser additions and removals will suffice because the number of crossings is, by definition, minimal. If each of the arcs on the right-hand side has no interior riser marks, then a minimal drawing of them will have no crossings.

For the moves of columns $3$ and $4$, we note that welded and virtual words may be replaced freely by alternate words because their level diagrams consist solely of their endpoints and trail and so they are isotopic to every path from their start to their end. This handles all of the cases involving $2$ or more virtual or welded crossings. We note that the remaining moves may be constructed in this same way because the trails can be cut sufficiently to allow this freedom for all non-virtual and non-welded crossings. Figure \ref{F:forbid} shows moves that are forbidden, and it is not obvious that these cannot be constructed.

To see that the forbidden moves cannot be constructed, we notice that because the central crossings are virtual or welded, then one of the trails involved must either be virtual or welded. As such, a riser mark must appear in the level diagram of the classical trail. This acts as a puncture to prevent any trail move that would be necessary to construct this move. If such a move is going to be constructed, it cannot be locally, and this is sufficient prevention.

The last move to construct is $g1$. Figure \ref{F:rigflip} shows how to do this by three riser moves and a rigid flip.

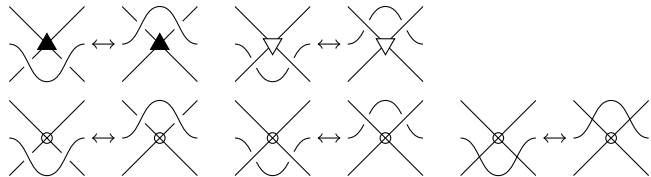
\begin{figure}[ht]
\begin{center}
\begin{tikzpicture}[every node/.style={scale=.5}, scale=.5]
	\begin{scope}[cm={1,0,0,1,(0,0)}]
	\draw[name path=a1, draw=none] (0, 2) -- (2, 0);
	\draw[name path=a2, draw=none] (0, 0) -- (2, 2);
	\draw[name path=a3, draw=none] (0, 1) .. controls (.5, 1) and (.5, 0) .. (1, 0) .. controls (1.5, 0) and (1.5, 1) .. (2, 1);
	\draw (0, 2) -- (2, 0);
	\draw (0, 0) -- (2, 2);
	\draw[name intersections={of=a2 and a3}] (intersection-1) node[circle, draw=white, fill=white, inner sep=1mm] {};
	\draw[name intersections={of=a1 and a3}] (intersection-1) node[circle, draw=white, fill=white, inner sep=1mm] {};
	\draw[name intersections={of=a1 and a2}] (intersection-1) node[regular polygon, regular polygon sides=3, draw=black, fill=black, inner sep=1mm] {};
	\draw (0, 1) .. controls (.5, 1) and (.5, 0) .. (1, 0) .. controls (1.5, 0) and (1.5, 1) .. (2, 1);
	\draw[name path=a1, draw=none] (3, 2) -- (5, 0);
	\draw[name path=a2, draw=none] (3, 0) -- (5, 2);
	\draw[name path=a3, draw=none] (3, 1) .. controls (3.5, 1) and (3.5, 2) .. (4, 2) .. controls (4.5, 2) and (4.5, 1) .. (5, 1);
	\draw (3, 2) -- (5, 0);
	\draw (3, 0) -- (5, 2);
	\draw[name intersections={of=a2 and a3}] (intersection-1) node[circle, draw=white, fill=white, inner sep=1mm] {};
	\draw[name intersections={of=a1 and a3}] (intersection-1) node[circle, draw=white, fill=white, inner sep=1mm] {};
	\draw[name intersections={of=a1 and a2}] (intersection-1) node[regular polygon, regular polygon sides=3, draw=black, fill=black, inner sep=1mm] {};
	\draw (3, 1) .. controls (3.5, 1) and (3.5, 2) .. (4, 2) .. controls (4.5, 2) and (4.5, 1) .. (5, 1);
	\draw[<->] (2.2, 1) -- (2.8, 1);
	\end{scope}
	\begin{scope}[cm={1,0,0,1,(6,0)}]
	\draw[name path=a1, draw=none] (0, 2) -- (2, 0);
	\draw[name path=a2, draw=none] (0, 0) -- (2, 2);
	\draw[name path=a3, draw=none] (0, 1) .. controls (.5, 1) and (.5, 0) .. (1, 0) .. controls (1.5, 0) and (1.5, 1) .. (2, 1);
	\draw (0, 1) .. controls (.5, 1) and (.5, 0) .. (1, 0) .. controls (1.5, 0) and (1.5, 1) .. (2, 1);
	\draw[name intersections={of=a2 and a3}] (intersection-1) node[circle, draw=white, fill=white, inner sep=1mm] {};
	\draw[name intersections={of=a1 and a3}] (intersection-1) node[circle, draw=white, fill=white, inner sep=1mm] {};
	\draw (0, 2) -- (2, 0);
	\draw (0, 0) -- (2, 2);
	\draw[name intersections={of=a1 and a2}] (intersection-1) node[regular polygon, regular polygon sides=3, rotate=180, draw=black, fill=white, inner sep=1mm] {};
	\draw[name path=a1, draw=none] (3, 2) -- (5, 0);
	\draw[name path=a2, draw=none] (3, 0) -- (5, 2);
	\draw[name path=a3, draw=none] (3, 1) .. controls (3.5, 1) and (3.5, 2) .. (4, 2) .. controls (4.5, 2) and (4.5, 1) .. (5, 1);
	\draw (3, 1) .. controls (3.5, 1) and (3.5, 2) .. (4, 2) .. controls (4.5, 2) and (4.5, 1) .. (5, 1);
	\draw[name intersections={of=a2 and a3}] (intersection-1) node[circle, draw=white, fill=white, inner sep=1mm] {};
	\draw[name intersections={of=a1 and a3}] (intersection-1) node[circle, draw=white, fill=white, inner sep=1mm] {};
	\draw (3, 2) -- (5, 0);
	\draw (3, 0) -- (5, 2);
	\draw[name intersections={of=a1 and a2}] (intersection-1) node[regular polygon, regular polygon sides=3, rotate=180, draw=black, fill=white, inner sep=1mm] {};
	\draw[<->] (2.2, 1) -- (2.8, 1);
	\end{scope}
	\begin{scope}[cm={1,0,0,1,(0,-2.5)}]
	\draw[name path=a1, draw=none] (0, 2) -- (2, 0);
	\draw[name path=a2, draw=none] (0, 0) -- (2, 2);
	\draw[name path=a3, draw=none] (0, 1) .. controls (.5, 1) and (.5, 0) .. (1, 0) .. controls (1.5, 0) and (1.5, 1) .. (2, 1);
	\draw (0, 2) -- (2, 0);
	\draw (0, 0) -- (2, 2);
	\draw[name intersections={of=a2 and a3}] (intersection-1) node[circle, draw=white, fill=white, inner sep=1mm] {};
	\draw[name intersections={of=a1 and a3}] (intersection-1) node[circle, draw=white, fill=white, inner sep=1mm] {};
	\draw[name intersections={of=a1 and a2}] (intersection-1) node[circle, draw=black, fill=none, inner sep=1mm] {};
	\draw (0, 1) .. controls (.5, 1) and (.5, 0) .. (1, 0) .. controls (1.5, 0) and (1.5, 1) .. (2, 1);
	\draw[name path=a1, draw=none] (3, 2) -- (5, 0);
	\draw[name path=a2, draw=none] (3, 0) -- (5, 2);
	\draw[name path=a3, draw=none] (3, 1) .. controls (3.5, 1) and (3.5, 2) .. (4, 2) .. controls (4.5, 2) and (4.5, 1) .. (5, 1);
	\draw (3, 2) -- (5, 0);
	\draw (3, 0) -- (5, 2);
	\draw[name intersections={of=a2 and a3}] (intersection-1) node[circle, draw=white, fill=white, inner sep=1mm] {};
	\draw[name intersections={of=a1 and a3}] (intersection-1) node[circle, draw=white, fill=white, inner sep=1mm] {};
	\draw[name intersections={of=a1 and a2}] (intersection-1) node[circle, draw=black, fill=none, inner sep=1mm] {};
	\draw (3, 1) .. controls (3.5, 1) and (3.5, 2) .. (4, 2) .. controls (4.5, 2) and (4.5, 1) .. (5, 1);
	\draw[<->] (2.2, 1) -- (2.8, 1);
	\end{scope}
	\begin{scope}[cm={1,0,0,1,(6,-2.5)}]
	\draw[name path=a1, draw=none] (0, 2) -- (2, 0);
	\draw[name path=a2, draw=none] (0, 0) -- (2, 2);
	\draw[name path=a3, draw=none] (0, 1) .. controls (.5, 1) and (.5, 0) .. (1, 0) .. controls (1.5, 0) and (1.5, 1) .. (2, 1);
	\draw (0, 1) .. controls (.5, 1) and (.5, 0) .. (1, 0) .. controls (1.5, 0) and (1.5, 1) .. (2, 1);
	\draw[name intersections={of=a2 and a3}] (intersection-1) node[circle, draw=white, fill=white, inner sep=1mm] {};
	\draw[name intersections={of=a1 and a3}] (intersection-1) node[circle, draw=white, fill=white, inner sep=1mm] {};
	\draw (0, 2) -- (2, 0);
	\draw[name intersections={of=a1 and a2}] (intersection-1) node[circle, draw=black, fill=none, inner sep=1mm] {};
	\draw (0, 0) -- (2, 2);
	\draw[name path=a1, draw=none] (3, 2) -- (5, 0);
	\draw[name path=a2, draw=none] (3, 0) -- (5, 2);
	\draw[name path=a3, draw=none] (3, 1) .. controls (3.5, 1) and (3.5, 2) .. (4, 2) .. controls (4.5, 2) and (4.5, 1) .. (5, 1);
	\draw (3, 1) .. controls (3.5, 1) and (3.5, 2) .. (4, 2) .. controls (4.5, 2) and (4.5, 1) .. (5, 1);
	\draw[name intersections={of=a2 and a3}] (intersection-1) node[circle, draw=white, fill=white, inner sep=1mm] {};
	\draw[name intersections={of=a1 and a3}] (intersection-1) node[circle, draw=white, fill=white, inner sep=1mm] {};
	\draw (3, 2) -- (5, 0);
	\draw[name intersections={of=a1 and a2}] (intersection-1) node[circle, draw=black, fill=none, inner sep=1mm] {};
	\draw (3, 0) -- (5, 2);
	\draw[<->] (2.2, 1) -- (2.8, 1);
	\end{scope}
	\begin{scope}[cm={1,0,0,1,(12,-2.5)}]
	\draw[name path=a1, draw=none] (0, 2) -- (2, 0);
	\draw[name path=a2, draw=none] (0, 0) -- (2, 2);
	\draw[name path=a3, draw=none] (0, 1) .. controls (.5, 1) and (.5, 0) .. (1, 0) .. controls (1.5, 0) and (1.5, 1) .. (2, 1);
	\draw (0, 1) .. controls (.5, 1) and (.5, 0) .. (1, 0) .. controls (1.5, 0) and (1.5, 1) .. (2, 1);
	\draw (0, 2) -- (2, 0);
	\draw[name intersections={of=a1 and a2}] (intersection-1) node[circle, draw=black, fill=none, inner sep=1mm] {};
	\draw (0, 0) -- (2, 2);
	\draw[name path=a1, draw=none] (3, 2) -- (5, 0);
	\draw[name path=a2, draw=none] (3, 0) -- (5, 2);
	\draw[name path=a3, draw=none] (3, 1) .. controls (3.5, 1) and (3.5, 2) .. (4, 2) .. controls (4.5, 2) and (4.5, 1) .. (5, 1);
	\draw (3, 1) .. controls (3.5, 1) and (3.5, 2) .. (4, 2) .. controls (4.5, 2) and (4.5, 1) .. (5, 1);
	\draw (3, 2) -- (5, 0);
	\draw[name intersections={of=a1 and a2}] (intersection-1) node[circle, draw=black, fill=none, inner sep=1mm] {};
	\draw (3, 0) -- (5, 2);
	\draw[<->] (2.2, 1) -- (2.8, 1);
	\end{scope}
\end{tikzpicture}
\caption{Forbidden local moves}\label{F:forbid}
\end{center}
\end{figure}

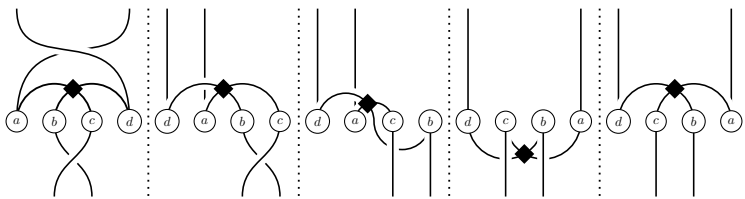
\begin{figure}[ht]
\begin{center}
\begin{tikzpicture}[scale=.5, every node/.style={scale=.5}]

	\begin{scope}
	\draw[double=black,draw=white,line width=2pt] (0, 0) .. controls (0, 3) and (3, 1) .. (3, 3);
	\draw[double=black,draw=white,line width=2pt] (3, 0) .. controls (3, 3) and (0, 1) .. (0, 3);
	\draw[double=black,draw=white,line width=2pt] (1, 0) .. controls (1, -1) and (2, -1) .. (2, -2);
	\draw[double=black,draw=white,line width=2pt] (2, 0) .. controls (2, -1) and (1, -1) .. (1, -2);
	\draw[line width=.7pt, name path=a1] (0, 0) arc[start angle=180, end angle=0, radius=1];
	\draw[line width=.7pt, name path=a2] (1, 0) arc[start angle=180, end angle=0, radius=1];
	\draw[name intersections={of=a1 and a2}, line width=.7pt] (intersection-1) node[draw=black, fill=black, diamond] {};

	\node[circle, draw=black, fill=white] (a) at (0, 0) {$a$};
	\node[circle, draw=black, fill=white] (b) at (1, 0) {$b$};
	\node[circle, draw=black, fill=white] (c) at (2, 0) {$c$};
	\node[circle, draw=black, fill=white] (d) at (3, 0) {$d$};
	\end{scope}

	\draw[line width=.7pt, dotted] (3.5, 3) -- (3.5, -2);
	\draw[line width=.7pt, dotted] (7.5, 3) -- (7.5, -2);
	\draw[line width=.7pt, dotted] (11.5, 3) -- (11.5, -2);
	\draw[line width=.7pt, dotted] (15.5, 3) -- (15.5, -2);
	\begin{scope}[cm={1,0,0,1,(3,0)}]
	
	\draw[double=black,draw=white,line width=2pt] (2, 0) -- (2, 3);
	\draw[double=black,draw=white,line width=2pt] (1, 0) -- (1, 3);
	\draw[double=black,draw=white,line width=2pt, name path=a1] (1, 0) arc[start angle=180, end angle=0, radius=1];
	\draw[double=black,draw=white,line width=2pt, name path=a2] (2, 0) arc[start angle=180, end angle=0, radius=1];
	\draw[double=black,draw=white,line width=2pt] (3, 0) .. controls (3, -1) and (4, -1) .. (4, -2);
	\draw[double=black,draw=white,line width=2pt] (4, 0) .. controls (4, -1) and (3, -1) .. (3, -2);
	\draw[name intersections={of=a1 and a2}, line width=.7pt] (intersection-1) node[draw=black, fill=black, diamond] {};

	\node[circle, draw=black, fill=white] (a) at (2, 0) {$a$};
	\node[circle, draw=black, fill=white] (b) at (3, 0) {$b$};
	\node[circle, draw=black, fill=white] (c) at (4, 0) {$c$};
	\node[circle, draw=black, fill=white] (d) at (1, 0) {$d$};
	\end{scope}

	\begin{scope}[cm={1,0,0,1,(8,0)}]

	\draw[double=black,draw=white,line width=2pt] (1, 0) -- (1, 3);
	\draw[double=black,draw=white,line width=2pt] (0, 0) -- (0, 3);
	\draw[double=black,draw=white,line width=2pt, name path=a1] (0, 0) arc[start angle=180, end angle=0, radius=.75] arc[radius=.75, start angle=-180, end angle=0];
	\draw[double=black,draw=white,line width=2pt, name path=a2] (1, 0) arc[start angle=180, end angle=0, radius=.5];
	\draw[double=black,draw=white,line width=2pt] (2, 0) -- (2, -2);
	\draw[double=black,draw=white,line width=2pt] (3, 0) -- (3, -2);
	\draw[name intersections={of=a1 and a2}, line width=.7pt] (intersection-1) node[draw=black, fill=black, diamond] {};

	\node[circle, draw=black, fill=white] (a) at (1, 0) {$a$};
	\node[circle, draw=black, fill=white] (b) at (3, 0) {$b$};
	\node[circle, draw=black, fill=white] (c) at (2, 0) {$c$};
	\node[circle, draw=black, fill=white] (d) at (0, 0) {$d$};
	
	\end{scope}

	\begin{scope}[cm={1,0,0,1,(12,0)}]

	\draw[double=black,draw=white,line width=2pt] (3, 0) -- (3, 3);
	\draw[double=black,draw=white,line width=2pt] (0, 0) -- (0, 3);
	\draw[double=black,draw=white,line width=2pt, name path=a1] (0, 0) arc[start angle=-180, end angle=0, radius=1];
	\draw[double=black,draw=white,line width=2pt, name path=a2] (1, 0) arc[start angle=-180, end angle=0, radius=1];
	\draw[double=black,draw=white,line width=2pt] (2, 0) -- (2, -2);
	\draw[double=black,draw=white,line width=2pt] (1, 0) -- (1, -2);
	\draw[name intersections={of=a1 and a2}, line width=.7pt] (intersection-1) node[draw=black, fill=black, diamond] {};

	\node[circle, draw=black, fill=white] (a) at (3, 0) {$a$};
	\node[circle, draw=black, fill=white] (b) at (2, 0) {$b$};
	\node[circle, draw=black, fill=white] (c) at (1, 0) {$c$};
	\node[circle, draw=black, fill=white] (d) at (0, 0) {$d$};
	
	\end{scope}

	\begin{scope}[cm={1,0,0,1,(16,0)}]

	\draw[double=black,draw=white,line width=2pt] (3, 0) -- (3, 3);
	\draw[double=black,draw=white,line width=2pt] (0, 0) -- (0, 3);
	\draw[double=black,draw=white,line width=2pt, name path=a1] (0, 0) arc[start angle=180, end angle=0, radius=1];
	\draw[double=black,draw=white,line width=2pt, name path=a2] (1, 0) arc[start angle=180, end angle=0, radius=1];
	\draw[double=black,draw=white,line width=2pt] (2, 0) -- (2, -2);
	\draw[double=black,draw=white,line width=2pt] (1, 0) -- (1, -2);
	\draw[name intersections={of=a1 and a2}, line width=.7pt] (intersection-1) node[draw=black, fill=black, diamond] {};

	\node[circle, draw=black, fill=white] (a) at (3, 0) {$a$};
	\node[circle, draw=black, fill=white] (b) at (2, 0) {$b$};
	\node[circle, draw=black, fill=white] (c) at (1, 0) {$c$};
	\node[circle, draw=black, fill=white] (d) at (0, 0) {$d$};
	
	\end{scope}
\end{tikzpicture}\caption{Constructing the rigid vertex flip}\label{F:rigflip}
\end{center}
\end{figure}
\end{proof}

\subsection{A representation whose size depends only on the number of words}

It may be that the tar sentence of a particular diagram is long but that the number of words in the sentence is low. In this case, it would be convenient to have a representation that exploits this fact. We show such a representation here. To achieve this representation we take the following steps.

\begin{enumerate}
\item Rename the labels so that the circle word is \texttt{;1.2.3.}$\cdots n$\texttt{.}. We regard the labels as having the corresponding numeric values.
\item Perform duplicate removals repeatedly until no more duplicate removals are possible.
\item Replace the first label, $l$, of each basic word with \texttt{*} $l$.
\item Replace the last label, $l$, of each basic word with \texttt{*} $l$.
\item We now regard the character \texttt{*} as part of the label that follows it.
\item Place a \texttt{,} between each pair of adjacent basic words.
\item Let $a$, $b$ be arbitrary labels. Replace each subword \texttt{+}$a$$b$ with $a$$b$\texttt{|-}$b$ if $a$ < $b$ and otherwise with $b$$a$\texttt{|-}$b$. Similarly, replace each subword \texttt{-}$a$$b$ with $a$$b$\texttt{|+}$b$ if $b$ < $a$ and otherwise with $b$$a$\texttt{|+}$b$.
\item Shorten any phrase \texttt{|+}$a$$w$ or \texttt{|-}$a$$w$ to $w$ where $a$ is a label and $w$ is one of \texttt{, \{ \} ( ) ! ;}.
\item Repeat the previous two steps until no further changes have been made. The arcs of the transformed trails are now separated by \texttt{|} characters, and the transformed trails are separated by other punctuation. 
\item In a sequence of $n$ square matrices with all entries $0$ and dimension $2n$, modify the matrices as follows.
\begin{enumerate}
\item{For every pair of consecutive labels $i$, $j$ let $r = 2i$. If $i$ contains a \texttt{*}, subtract $1$ from $r$. Let $c$ be the analogous number for $j$.}
\item{If $ij$ arose from the transformation of the $k$th basic word, add $1$ to the $r$th row and $c$th column of the $k$th matrix.}
\end{enumerate}
\item Let the entry $M_{1,1}$ of each matrix be determined by its type, with $0$ classical, $1$ virtual, $2$ welded, $3$ rigid, and higher numbers used for flat crossing hierarchies.
\item The sequence of matrices is a representation of the same link as the tar sentence from which it was constructed.
\end{enumerate}

\begin{proposition}The sequence of matrices given can be used to construct a tar sentence, up to orientation changes of the components.
\end{proposition}

\begin{proof}There are two key facts. One is that trails are self-avoiding. From this, we may reconstruct the basic words by drawing the arcs in the hemispheres of $S^2$ disjointly and then connecting the ends. Each basic word is represented by one matrix. For each matrix, we draw the curve by dividing the circle into $n$ pieces. We label the pieces with odd numbers $1, 3, \dots, 2n-1$ and the interstices with even numbers $2, 4, \dots, 2n$. Entries in the upper triangle of the matrix represent arcs in the upper hemisphere; entries in the lower triangle of the matrix represent arcs in the lower hemisphere. In either case, the arcs are drawn so that they project as chords from the region indicated by the row to the region indicated by the column. From this curve, either the basic word or its reverse may be read.

We recover the order of the basic words from the order of the matrices. We recover the type of the word, and any flat set hierarchy, from the upper left entry. We recover the basic word or its reverse by the process just described. The circle word is given by \texttt{;1.2.}$\cdots n$\texttt{.}. We may reverse at most half of the basic words so that the orientation of the link components is consistent. After this, the entire tar sentence, up to orientation change, is recovered.
\end{proof}

\begin{remark}If $k$ trails of the same type do not cross, then the sum of their matrices gives back these $k$ curves upon reconstruction. Using this technique, we may represent classical links in only $2$ matrices -- one for bridges, one for underpasses. Suitably transformed diagrams of other types, with the notable exception of those containing flat hierarchies or rigid crossings, may likewise be represented by a small number of matrices.
\end{remark}

\section{Two examples where the notation is useful}

We give two examples of the utility of the notation.

\subsection{The sentence-style notation: the Kishino virtual knot is non-classical}
In the first we use the sentence style of the notation. We examine the Kishino virtual knot of Kishino and Saito \cite{kishinoknot} and provide another proof that it is non-trivial and non-classical by use of the tar sentence and its moves. Figure \ref{F:kishino} shows a diagram of the knot and a corresponding tar sentence.

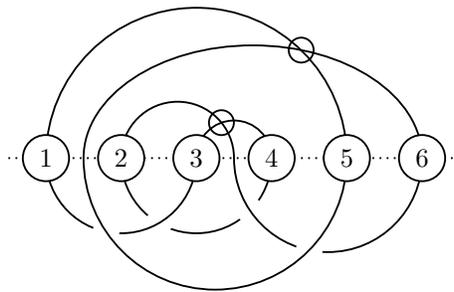
\begin{figure}[ht]
\begin{center}
\begin{tikzpicture}[line width=.7pt]
	\draw[draw=none, name path=a1] (2, 0) arc[start angle=180, end angle=0, radius=.5];
	\draw[draw=none, name path=a2] (0, 0) arc[start angle=180, end angle=0, radius=2];
	\draw[draw=none, name path=a3] (5, 0) arc[start angle=0, end angle=180, y radius=1.5, x radius=2.25] arc[start angle=-180, end angle=0, radius=1.75];
	\draw[draw=none, name path=a4] (1, 0) arc[start angle=180, end angle=0, radius=.75] arc[start angle=-180, end angle=0, radius=1.25];
	\draw[draw=none, name path=a5] (0, 0) arc[start angle=-180, end angle=0, radius=1];
	\draw[draw=none, name path=a6] (1, 0) arc[start angle=-180, end angle=0, radius=1];

	\draw (1, 0) arc[start angle=-180, end angle=0, radius=1];
	\draw[name intersections={of=a5 and a6}] (intersection-1) node[circle,draw=white,fill=white] {};
	\draw[name intersections={of=a4 and a6}] (intersection-2) node[circle,draw=white,fill=white] {};
	\draw (0, 0) arc[start angle=-180, end angle=0, radius=1];
	\draw[name intersections={of=a3 and a5}] (intersection-1) node[circle,draw=white,fill=white] {};
	\draw (1, 0) arc[start angle=180, end angle=0, radius=.75] arc[start angle=-180, end angle=0, radius=1.25];
	\draw[name intersections={of=a3 and a4}] (intersection-2) node[circle,draw=white,fill=white] {};
	\draw[name intersections={of=a4 and a1}] (intersection-1) node[circle,draw=black,fill=none] {};
	\draw (5, 0) arc[start angle=0, end angle=180, y radius=1.5, x radius=2.25] arc[start angle=-180, end angle=0, radius=1.75];
	\draw[name intersections={of=a2 and a3}] (intersection-1) node[circle,draw=black,fill=none] {};
	\draw (0, 0) arc[start angle=180, end angle=0, radius=2];
	\draw (2, 0) arc[start angle=180, end angle=0, radius=.5];
	
	\draw[dotted] (-.5, 0) -- (5.5, 0);
	\begin{scope}[every node/.style={circle,draw=black,fill=white}]
	\node at (1, 0) {$2$};
	\node at (5, 0) {$6$};
	\end{scope}
	\begin{scope}[every node/.style={circle,draw=black,fill=white}]
	\node at (0, 0) {$1$};
	\node at (4, 0) {$5$};
	\node at (2, 0) {$3$};
	\node at (3, 0) {$4$};
	\end{scope}
	
	\node at (2.5, -2.5) {As virtual: \texttt{[+3.4.][+5.1.]+6.1.5.+2.3.6.-1.3.-4.2.;1.2.3.4.5.6.}};
	\node at (2.5, -2.9) {As flat virtual: \texttt{[+3.4.][+5.1.]\{+6.1.5.+2.3.6.-1.3.-4.2.\};1.2.3.4.5.6.}};
\end{tikzpicture}
\caption{A tar sentence of the Kishino virtual knot}\label{F:kishino}
\end{center}
\end{figure}

\begin{proposition}The Kishino virtual knot is non-classical.
\end{proposition}

\begin{proof}We proceed by examining the level diagrams and considering moves. The level diagrams are shown in Fig. \ref{F:kishinoproof}.

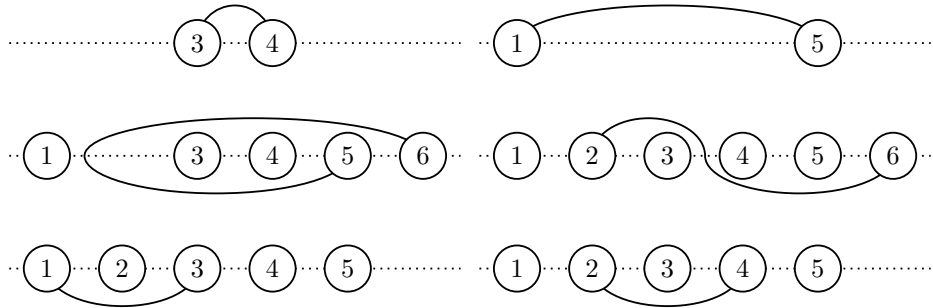
\begin{figure}[ht]
\begin{center}
\begin{tikzpicture}[line width=.7pt]

	\begin{scope}
		\draw (2, 0) arc[start angle=180, end angle=0, radius=.5];
		\draw[dotted] (-.5, 0) -- (5.5, 0);
		\begin{scope}[every node/.style={circle,draw=black,fill=white}]
		\end{scope}
		\begin{scope}[every node/.style={circle,draw=black,fill=white}]
		\node at (2, 0) {$3$};
		\node at (3, 0) {$4$};
		\end{scope}
	\end{scope}

	\begin{scope}[cm={1,0,0,1,(6.25, 0)}]
		\draw (0, 0) arc[start angle=180, end angle=0, y radius=.5, x radius=2];
		\draw[dotted] (-.5, 0) -- (5.5, 0);
		\begin{scope}[every node/.style={circle,draw=black,fill=white}]
		\node at (0, 0) {$1$};
		\node at (4, 0) {$5$};
		\end{scope}
	\end{scope}	

	\begin{scope}[cm={1,0,0,1,(0,-1.5)}]
	\draw (5, 0) arc[start angle=0, end angle=180, y radius=.5, x radius=2.25] arc[start angle=-180, end angle=0, y radius=.5, x radius=1.75];
	
	\draw[dotted] (-.5, 0) -- (5.5, 0);
		\begin{scope}[every node/.style={circle,draw=black,fill=white}]
		\node at (5, 0) {$6$};
		\node at (0, 0) {$1$};
		\node at (4, 0) {$5$};
		\node at (2, 0) {$3$};
		\node at (3, 0) {$4$};
		\end{scope}
	\end{scope}

	\begin{scope}[cm={1,0,0,1,(6.25, -1.5)}]
		\draw (1, 0) arc[start angle=180, end angle=0, y radius=.5, x radius=.75] arc[start angle=-180, end angle=0, y radius=.5, x radius=1.25];
		\draw[dotted] (-.5, 0) -- (5.5, 0);
		\begin{scope}[every node/.style={circle,draw=black,fill=white}]
		\node at (1, 0) {$2$};
		\node at (5, 0) {$6$};
		\node at (0, 0) {$1$};
		\node at (4, 0) {$5$};
		\node at (2, 0) {$3$};
		\node at (3, 0) {$4$};
		\end{scope}
	\end{scope}

	\begin{scope}[cm={1,0,0,1,(0,-3)}]
		\draw (0, 0) arc[start angle=-180, end angle=0, y radius=.5, x radius=1];
		\draw[dotted] (-.5, 0) -- (5.5, 0);
		\begin{scope}[every node/.style={circle,draw=black,fill=white}]
		\node at (0, 0) {$1$};
		\node at (1, 0) {$2$};
		\node at (4, 0) {$5$};
		\node at (2, 0) {$3$};
		\node at (3, 0) {$4$};
		\end{scope}
	\end{scope}

	\begin{scope}[cm={1,0,0,1,(6.25,-3)}]
		\draw (1, 0) arc[start angle=-180, end angle=0, y radius=.5, x radius=1];
		\draw[dotted] (-.5, 0) -- (5.5, 0);
		\begin{scope}[every node/.style={circle,draw=black,fill=white}]
		\node at (1, 0) {$2$};
		\node at (0, 0) {$1$};
		\node at (4, 0) {$5$};
		\node at (2, 0) {$3$};
		\node at (3, 0) {$4$};
		\end{scope}
	\end{scope}
\end{tikzpicture}
\caption{Level diagrams of the Kishino virtual knot}\label{F:kishinoproof}.
\end{center}
\end{figure}

We see that neither of the virtual trails may be moved by a trail move to avoid all of its crossings, nor may we place all of the virtual crossings on one trail. This is easily seen by noticing that the trails \texttt{+2.3.6.} and \texttt{-4.2.} form a circle around the riser mark $3$, and the trails \texttt{+6.1.5.} and \texttt{+2.3.6.} form a circle around the riser mark $5$. The ends of the virtual trails therefore lie in three distinct regions and cannot be made to lie on one trail because then the ends would only lie in two distinct regions.

Because the ends of the virtual trails appear on every classical level diagram, no riser move can free them from these regions, and no trail move or riser addition can change the separation of the virtual ends. Since there are two virtual crossings and must be at least two virtual trails, the number of virtual words in the sentence cannot be reduced. A link of which every tar sentence has more than $0$ virtual words cannot be classical.\end{proof}

\subsection{The reduced notation style: a $3$-bridge diagram of knot $11a1$}

Figure \ref{F:k11a1r} shows knot $11a1$ (we adopt the numbering scheme of \cite{knotatlas}) in a minimal crossing form. Figure \ref{F:k11a1b} shows this same knot in a $3$-bridge, $64$-crossing form. Because the matrices produced by the reduced notation style representation are sparse, we list only the non-zero entries of these, and we do so in symbolic form. That is, we list the corresponding sentence labels rather than the matrix coordinates. Because the representation is in the form of bridges and underpasses, we sum together the three bridge matrices and also sum together the three underpass matrices.

The result of this is the representation:
\begin{align*}
\mbox{Bridges: } &(\texttt{*}1, 5) = 1, (1, 2) = 7, (1,4) = 2, (1, \texttt{*}5) = 1, (1, 5) = 13\\
	&(2,4) = 3, (\texttt{*}3, 4) = 1, (3, 4) = 1, (5, 6) = 8, (3, 2) = 1\\
	&(\texttt{*}4, 2) = 1, (4, 2) = 4, (5, 1) = 14, (5, \texttt{*}2) = 1, (5, 2) = 4\\
	&(5, 4) = 3, (\texttt{*}6, 1) = 1, (6, 1) = 8\\
\mbox{Underpasses: } &(\texttt{*}2, \texttt{*}3) = 1, (\texttt{*}4, \texttt{*}1) = 1, (\texttt{*}5, \texttt{*}6) = 1.
\end{align*}
We may read this representation directly from the diagram by labeling the leftmost point on the central underpass $1$, continuing clockwise to choose $2$ as the leftmost of the lower underpass, $3$ as the rightmost of that same underpass, $4$ as the rightmost of the center underpass, $5$ as the rightmost of the upper underpass, and $6$ as the leftmost of the upper underpass. The binding circle is drawn as straight lines, and the upper and lower underpass are pushed slightly out of the interior of the circle.

\begin{figure}
\begin{center}
\begin{tikzpicture}[x=12pt, y=12pt, line width=.7pt, every node/.style={draw=white,fill=white,circle,inner sep=.5mm}]
    \draw (2,0) -- (2, -1) -- (4,-1);
    \node at (3, -1) {};
    \draw (5,-4) -- (5,-2.5);
    \node at (5, -3) {};
    \draw (6.5,0) -- (8,0);
    \node at (7, 0) {};
    \draw (4,0) -- (5.5,0);
    \node at (5, 0) {};
    \draw (0,0) -- (0, -3) -- (9, -3) -- (9, 1) -- (7, 1) -- (7, -2) -- (3, -2) -- (3, 1) -- (1, 1) -- (1,0);
    \node at (1, -3) {};
    \node at (8, 1) {};
    \node at (5, -2) {};
    \node at (2, 1) {};
    \draw (5,-2.5) -- (5, 2) -- (2, 2) -- (2,0);
    \node at (5, -1) {};
    \node at (5, 1) {};
    \draw (1,0) -- (1, -4) -- (5,-4);
    \draw (8,0) -- (8, 3) -- (0, 3) -- (0,0);
    \draw (4,-1) -- (5, -1) -- (6, -1) -- (6, 1) -- (4, 1) -- (4,0);
    \node at (6, 0) {};
    \draw (5.5,0) -- (6.5,0);
\end{tikzpicture}
\caption{Knot $11a1$ in an $11$-crossing alternating projection.}\label{F:k11a1r}
\end{center}
\end{figure}
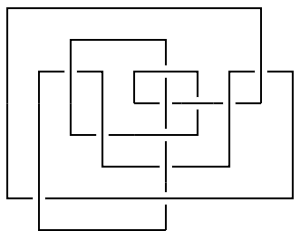

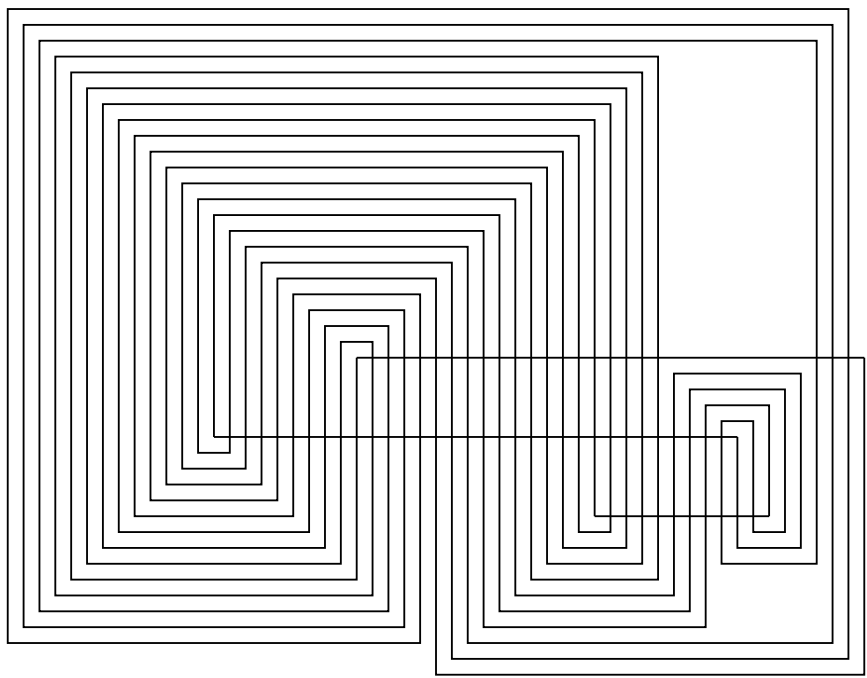
\begin{figure}
\begin{center}
\begin{tikzpicture}[x=6pt, y=6pt, draw=black, line width=.7pt]
    \draw (75, 9) -- (43, 9);
    \draw (69, -1) -- (58, -1);
    \draw (34, 4) -- (67, 4);    
    \draw (43, 9) -- (43, -5) -- (25, -5) -- (25, 27) -- (61, 27) -- (61, -4) -- (55, -4) -- (55, 21) -- (31, 21) -- (31, 1) -- (37, 1) -- (37, 15) -- (49, 15) -- (49, -10) -- (74, -10) -- (74, 31) -- (21, 31) -- (21, -9) -- (47, -9) -- (47, 13) -- (39, 13) -- (39, -1) -- (29, -1) -- (29, 23) -- (57, 23) -- (57, -2) -- (59, -2) -- (59, 25) -- (27, 25) -- (27, -3) -- (41, -3) -- (41, 11) -- (45, 11) -- (45, -7) -- (23, -7) -- (23, 29) -- (72, 29) -- (72, -4) -- (66, -4) -- (66, 5) -- (68, 5) -- (68, -2) -- (70, -2) -- (70, 7) -- (64, 7) -- (64, -7) -- (52, -7) -- (52, 18) -- (34, 18) -- (34, 4);    
    \draw (58, -1) -- (58, 24) -- (28, 24) -- (28, -2) -- (40, -2) -- (40, 12) -- (46, 12) -- (46, -8) -- (22, -8) -- (22, 30) -- (73, 30) -- (73, -9) -- (50, -9) -- (50, 16) -- (36, 16) -- (36, 2) -- (32, 2) -- (32, 20) -- (54, 20) -- (54, -5) -- (62, -5) -- (62, 28) -- (24, 28) -- (24, -6) -- (44, -6) -- (44, 10) -- (42, 10) -- (42, -4) -- (26, -4) -- (26, 26) -- (60, 26) -- (60, -3) -- (56, -3) -- (56, 22) -- (30, 22) -- (30, 0) -- (38, 0) -- (38, 14) -- (48, 14) -- (48, -11) -- (75, -11) -- (75, 9);    
    \draw (67, 4) -- (67, -3) -- (71, -3) -- (71, 8) -- (63, 8) -- (63, -6) -- (53, -6) -- (53, 19) -- (33, 19) -- (33, 3) -- (35, 3) -- (35, 17) -- (51, 17) -- (51, -8) -- (65, -8) -- (65, 6) -- (69, 6) -- (69, -1);
\end{tikzpicture}
\caption{Knot $11a1$ in a $3$-bridge projection. At all crossings, the vertical strand is above.}\label{F:k11a1b}
\end{center}
\end{figure}

\section*{Acknowledgments}

This article is the result of discussions with many participants from the 2012 workshop "Invariants in Low-Dimensional Topology and Knot Theory" at the Mathematisches Forschungsinstitut Oberwolfach. The author is grateful to the institute and the participants. In particular, J. Scott Carter, Heather Dye, Slavik Jablan, Lou Kauffman, and Vassily Manturov each made specific suggestions that have been incorporated into this paper and improved its clarity. The anonymous reviewer also made helpful suggestions, and the author is grateful to him or her.
 
\bibliographystyle{hplain}
\bibliography{represent}

\begin{thebibliography}{1}

\bibitem{knotatlas}
Dror Bar-Natan, Scott Morrison, and et~al.
\newblock The {K}not {A}tlas.
\newblock http://katlas.org/.

\bibitem{welded97}
R.~Fenn, R.~Rimanyi, and C.~Rourke.
\newblock The braid permutation group.
\newblock {\em Topology}, 36(1):123--135, 1997.

\bibitem{isoebnf}
{ISO/IEC}.
\newblock {International Standard ISO/IEC 14977:1996 -- Information technology
  -- Syntactic metalanguage -- Extended BNF}, {1996}.

\bibitem{vkt99}
L.H. Kauffman.
\newblock Virtual knot theory.
\newblock {\em European J. Comb.}, 20:663--690, 1999.

\bibitem{kauffman06}
Louis~H. Kauffman.
\newblock {\em Formal Knot Theory}.
\newblock Dover, 2006 edition, 1983.

\bibitem{kishinoknot}
T.~Kishino and S.~Satoh.
\newblock A note on non-classical virtual knots.
\newblock {\em J. Knot Theory Ramifications}, 13(7):845--856, 2004.

\bibitem{murasugi}
Kunio Murasugi.
\newblock {\em Knot Theory \& Its Applications}.
\newblock Birkh\"{a}user, 1996.
\newblock English edition.

\bibitem{nelson}
Sam Nelson.
\newblock The combinatorial revolution in knot theory.
\newblock {\em Notices Amer. Math. Soc.}, 58(11):1553--1561, 2011.

\end{thebibliography}

\end{CJK*}\end{document}